\documentclass{amsart}

\usepackage{float}
\usepackage{amsmath}
\usepackage{amsfonts}
\usepackage{amssymb}
\usepackage{graphicx}
\usepackage{hyperref}
\usepackage{mathrsfs}
\usepackage{pb-diagram}
\usepackage{epstopdf}
\usepackage{amsmath,amsfonts,amsthm,enumerate,amscd,latexsym,curves}
\usepackage{bbm}
\usepackage{mathrsfs}
\usepackage{epsfig,epsf,xypic,epic}
\usepackage{caption}
\usepackage{subcaption}
\usepackage{tikz}
\usepackage{geometry}
\usepackage{enumitem}
\usetikzlibrary{matrix,arrows,decorations.pathmorphing}
\usetikzlibrary{cd}
\usepackage{tikz-cd}

\usetikzlibrary{matrix,arrows,decorations.pathmorphing}

\newtheorem{theorem}{Theorem}[section]

\newtheorem{lemma}[theorem]{Lemma}
\newtheorem{corollary}[theorem]{Corollary}

\newtheorem{definition}[theorem]{Definition}

\newtheorem{example}[theorem]{Example}

\newtheorem{proposition}[theorem]{Proposition}

\newtheorem{remark}[theorem]{Remark}

\newcommand{\Image}{\mathrm{Im}}
\newcommand{\Ker}{\mathrm{Ker}}

\newcommand{\tot}{\mathrm{Tot}}

\newcommand{\RD}{\mathrm{R}}
\newcommand{\rk}{\mathrm{rk}}
\newcommand{\Hom}{\mathrm{Hom}}
\newcommand{\Ext}{\mathrm{Ext}}
\newcommand{\Int}{\mathrm{Int}}

\newcommand{\inner}[1]{\langle #1\rangle}

\newcommand{\bb}[1]{\mathbb{#1}}
\newcommand{\cu}[1]{\mathcal{#1}}
\newcommand{\til}[1]{\widetilde{#1}}
\newcommand{\msc}[1]{\mathscr{#1}}
\newcommand{\ol}[1]{\overline{#1}}
\newcommand{\mf}[1]{\mathfrak{#1}}
\newcommand{\ul}[1]{\underline{#1}}
\newcommand{\Spec}[1]{\mathrm{Spec}({#1})}

\def\codim{\mathrm{codim}}

\begin{document}
	
	\title[]{Tropical Lagrangian multi-sections \\and\\ tropical locally free sheaves}
	\author[Suen]{Yat-Hin Suen}
	\address{Center for Geometry and Physics\\ Institute for Basic Science (IBS)\\ Pohang 37673\\ Republic of Korea}
	\email{yhsuen@ibs.re.kr}
	\date{\today}

	\begin{abstract}
    This article is a continuation of the work \cite{CMS_k3bundle}. We generalize the notion of tropical Lagrangian multi-sections to any dimensions. Together with some linear algebra data, we construct a special class of locally free sheaves, called tropical locally free sheaves. We will also provide the reverse construction and show that there is a 1-1 correspondence between isomorphism classes of tropical locally free sheaves and tropical Lagrangian multi-sections modulo certain equivalence.
	\end{abstract}
	
	\maketitle
	
	\tableofcontents
	
\section{Introduction}

The Gross-Siebert program \cite{GS1,GS2,GS11} gives an algebro-geometric understanding of SYZ mirror symmetry \cite{SYZ}. In \cite{CMS_k3bundle}, together with Chan and Ma, the author of this paper attempted to understand homological mirror symmetry \cite{HMS} in terms of the Gross-Siebert setup. We introduced there the notion of \emph{tropical Lagrangian multi-sections} over any 2-dimensional integral affine manifold of singularities $B$ equipped with polyhedral decomposition $\msc{P}$ and constructed, by fixing certain local model, a locally free sheaf $\cu{E}_0$ over the associated scheme $X_0(B,\msc{P})$. We also provided a nice combinatorial condition for smoothability of the pair $(X_0(B,\msc{P}),\cu{E}_0)$ under some extra assumptions.  

In this article, we will generalize the notion of tropical Lagrangian multi-sections to any dimension. We begin by reviewing some preliminary of the Gross-Siebert program in Section \ref{sec:GS}. In Section \ref{sec:trop_Lag}, we introduce the notion of tropical Lagrangian multi-sections over any integral affine manifold with singularities $B$ equipped with a polyhedral decomposition $\msc{P}$. For this purpose, we need the notion of \emph{tropical spaces}, which has been introduced in \cite{trop_spaces}. Roughly speaking, tropical spaces are spaces that allow us to talk about sheaf of affine and piecewise linear functions. A tropical Lagrangian multi-section $\bb{L}$ over $(B,\msc{P})$ is then a branched covering map $\pi:(L,\msc{P}',\mu)\to(B,\msc{P})$ between tropical spaces that respect the polyhedral decomposition $\msc{P},\msc{P}'$ together with a multi-valued piecewise linear function $\varphi$ on $L$. Here $\mu:\msc{P}\to\bb{Z}_{>0}$ is the multiplicity map. The multi-valued function $\varphi$ should be thought of as certain tropical limit of the local potential of a Lagrangian multi-section of the SYZ fibration.

Given a tropical Lagrangian multi-section $\bb{L}$ over $(B,\msc{P})$. Due to its discrete nature, one shouldn't expect $\bb{L}$ can determine a sheaf on $X_0(B,\msc{P})$ uniquely. Therefore, we need to prescribe some continuous data on top of the discrete data determined by $\bb{L}$. We will introduce in \ref{sec:constructing_E_0} two continuous data $({\bf{g}},{\bf{h}})$ which guarantee the existence of a locally free sheave on $X_0(B,\msc{P})$. The idea is to apply the technique in \cite{Suen_trop_lag} to construct a collection of toric vector bundles $\{\cu{E}({\bf{g}}(\tau'))\}_{\pi(\tau')=\tau}$ on each toric piece $X_{\tau}$ by using the multi-valued piecewise linear function $\varphi$ on $L$. The rank of each $\cu{E}({\bf{g}}(\tau'))$ is given by the ramification degree of the cell $\tau'\in\msc{P}'$. Put
$$\cu{E}({\bf{g}}(\tau)):=\bigoplus_{\pi(\tau')=\tau}\cu{E}({\bf{g}}(\tau')),$$
which is a toric vector bundle on $X_{\tau}$, whose rank is exactly the degree of the branched covering map $\pi:L\to B$. Then we glue the vector bundles $\{\cu{E}({\bf{g}}(\tau))\}_{\tau\in\msc{P}}$ together by using the data ${\bf{h}}$. However, in general, we may encounter an extra twisting data $\ol{s}$ when we glue the toric strata $\{X_{\tau}\}_{\tau\in\msc{P}}$ together. In this case, there is an obstruction $o_{\bb{L}}([\ol{s}])\in  H^2(L,\bb{C}^{\times})$ for gluing $\{\cu{E}({\bf{g}}(\tau))\}_{\tau\in\msc{P}}$. This obstruction generalizes the obstruction maps appear in \cite{GS1}, Theorem 2.34 (the ample line bundle case) and \cite{CMS_k3bundle}, Theorem 5.5 (the 2-dimensional case).

\begin{theorem}[=Theorem \ref{thm:close}]
    Suppose $\ol{s}$ is the associated closed gluing data of an open gluing data $s$. The locally free sheaves $\{\cu{E}({\bf{g}}(\tau))\}_{\tau\in\msc{P}}$ can be glued to a rank $r$ locally free sheaf on $X_0(B,\msc{P}, s)$ via the data $({\bf{g}},{\bf{h}})$ if and only if $o_{\bb{L}}([\ol{s}])=1$.
\end{theorem}

The vanishing of $o_{\bb{L}}([\ol{s}])$ will give us another continuous data ${\bf{k}}_s$, which can be combined with ${\bf{h}}$ to form ${\bf{h}}_s$. We then write ${\bf{D}}_s$ for the data $({\bf{g}},{\bf{h}}_s)$ and denote by $\msc{D}_s(\bb{L})$ the set of all such data. Theorem 1.1 says that a choice of data ${\bf{D}}_s\in\msc{D}_s(\bb{L})$ will provide us a locally free sheaf $\cu{E}_0(\bb{L},{\bf{D}}_s)$ on the scheme $X_0(B,\msc{P},s)$.

In \cite{FOOO1}, the notion of unobstructed Lagrangian submanifolds (\cite{AJ} for immersed Lagrangian submanifolds) was introduced. The main feature of an unobstructed Lagrangian submanifolds is that its Floer cohomology is well-defined and hence defines an object in the Fukaya category. In particular, unobstructed Lagrangian submanifolds should have the corresponding mirror objects. Therefore, we borrow this terminology here. Namely, for a fixed gluing data $s$, a tropical Lagrangian multi-section $\bb{L}$ over $(B,\msc{P})$ is said to be \emph{unobstructed} if $\msc{D}_s(\bb{L})\neq\emptyset$ and a pair $(\bb{L},{\bf{D}}_s)$ is called a \emph{tropical Lagrangian brane} (Definition \ref{def:unobstructed}).

Section \ref{sec:E_to_L} will be devoted to the reverse construction. Based on the construction in Section \ref{sec:constructing_E_0}, we introduce the notion of \emph{tropical locally free sheaves} (Definition \ref{def:tropical}). To such a locally free sheaf $\cu{E}_0$, we are able to construct a canonical tropical Lagrangian multi-section $\bb{L}_{\cu{E}_0}$. Moreover, there is a natural data ${\bf{D}}_s(\cu{E}_0)\in\msc{D}_s(\bb{L}_{\cu{E}_0})$ such that
$$\cu{E}_0\cong\cu{E}_0(\bb{L}_{\cu{E}_0},{\bf{D}}_s(\cu{E}_0)).$$
We will end Section \ref{sec:E_to_L} by showing that we actually have an abundant sources of examples given by restricting toric vector bundles on the toric boundary of a toric variety (Theorem \ref{thm:many_examples}).

Given $(\bb{L},{\bf{D}}_s)$, Section \ref{sec:constructing_E_0} has taught us how to construct a tropical locally free sheaf $\cu{E}_0:=\cu{E}_0(\bb{L},{\bf{D}}_s)$ while Section \ref{sec:E_to_L} has provided a canonical tropical Lagrangian brane $(\bb{L}_{\cu{E}_0},{\bf{D}}_s(\cu{E}_0))$ out of $\cu{E}_0$. It is natural to compare $(\bb{L},{\bf{D}}_s)$ and $\bb{L}_{\cu{E}_0}$. In general, we may not have $\bb{L}=\bb{L}_{\cu{E}_0}$. This indicates the phenomenon that non-Hamiltonian isotopic or topologically different Lagrangian submanifolds can still be equivalent to each other in the derived (immersed) Fukaya category. See, for example, \cite{CS_SYZ_imm_Lag}. Therefore, we would still like to regard them as the same object. This brings us to Section \ref{sec:equiv}, where we will introduce the notion of \emph{combinatorial equivalence} of tropical Lagrangian branes and prove the following

\begin{theorem}[=Theorem \ref{thm:bijection}]
    We have a canonical bijection
    $$\cu{F}:\frac{\{\text{Tropical locally free sheaves on }X_0(B,\msc{P},s)\}}{\text{isomorphism}}\to\frac{\{(\bb{L},{\bf{D}}_s)\,|\,\bb{L}\text{ being unobstructed and }{\bf{D}}_s\in\msc{D}_s(\bb{L})\}}{\text{combinatorial equivalence}},$$
    given by $\cu{E}_0\mapsto(\bb{L}_{\cu{E}_0},{\bf{D}}_s(\cu{E}_0))$. Its inverse is given by $(\bb{L},{\bf{D}}_s)\mapsto\cu{E}_0(\bb{L},{\bf{D}}_s)$.
\end{theorem}

Theorem 1.2 provides a slightly more geometric understanding of homological mirror symmetry in the sense that we don't need any derived objects on both sides to achieve the correspondence. Although a tropical Lagrangian multi-section is still not yet an honest Lagrangian multi-section, one should expect that an unobstructed Lagrangian multi-section can be constructed from the data $(\bb{L},{\bf{D}}_s)$. We left this for future research.

\section{The Gross-Siebert program}\label{sec:GS}

We give a brief review of how the scheme $X_0(B,\msc{P},s)$ is constructed. We follow \cite{GS1} and use the \emph{fan construction}.

Let $B$ be an integral affine manifold with singularities equipped with a polyhedral decomposition $\msc{P}$. Elements in $\msc{P}$ are celled \emph{cells}. Throughout the whole article, we assume $B$ is compact without boundary and all cells have no self-intersections. By taking the barycentric decomposition situation of $\msc{P}$, there is a canonical open cover $\cu{W}:=\{W_{\sigma}\}_{\sigma\in\msc{P}}$ of $B$ so that $W_{\tau}\cap W_{\sigma}\neq\emptyset$ if and only if $\tau\subset\sigma$. Denote by $\Delta\subset B$ the singular locus of $B$, which is a union of locally closed codimension 2 submanifolds inside the codimension 1 strata of $B$ and $\Lambda$ the lattice induced by the integral structure on $B\backslash\Delta$. For $\tau\in\msc{P}$, let
$$\Lambda_{\tau}:=\{v\in\Lambda_y:v\text{ is tangent to }\tau \text{ at }y\},$$
for any $y\in\Int(\tau)\backslash\Delta$. This lattice is independent of the choice of $y\in\Int(\tau)$. We also define
$$\cu{Q}_{\tau}:=\Lambda/\Lambda_{\tau}.$$
For $\tau\subset\sigma$, by parallel transport along a path in $B\backslash\Delta$ starting on $\Int(\tau)\backslash\Delta$ and ending on $\Int(\sigma)$, we get a projection $p:\cu{Q}_{\tau}\to\cu{Q}_{\sigma}$. We always assume $\msc{P}$ is \emph{toric}, that is, for each $\tau\in\msc{P}$, there is a submersion $S_{\tau}:W_{\tau}\to\cu{Q}_{\tau}\otimes_{\bb{Z}}\bb{R}$. This gives a complete fan
$$\Sigma_{\tau}:=\{r\cdot S_{\tau}(\sigma)\,|\,\sigma\supset\tau,r\geq 0\}$$
on $\cu{Q}_{\tau}\otimes_{\bb{Z}}\bb{R}$ and hence a complete toric variety $X_{\tau}$. To glue them together, Gross-Siebert introduced the category ${\textbf{Cat}}(\msc{P})$, whose objects are elements in $\msc{P}$ and
$$\Hom_{{\textbf{Cat}}(\msc{P})}(\tau,\sigma):=\begin{cases}
\emptyset & \text{ if }\tau\not\subset\sigma,\\
\{e\} & \text{ if }\tau\subset\sigma.
\end{cases}$$
For $e:\tau\to\sigma$, the inclusion $p_e^*:\cu{Q}_{\sigma}^*\to\cu{Q}_{\tau}^*$ induces a natural inclusion $F(e):X_{\sigma}\to X_{\tau}$ that satisfies
$$F(e_2\circ e_1)=F(e_1)\circ F(e_2),$$
for all $e_1:\sigma_1\to\sigma_2,e_2:\sigma_2\to\sigma_3$. Hence we can take the limit
$$\lim_{\longrightarrow}X_{\sigma}$$
to obtain an algebraic space. One can twist this construction by a cocycle $[\ol{s}]\in  H^1(\cu{W},\cu{Q}_{\msc{P}}\otimes\bb{C}^{\times})$. Such an extra twisting is called a \emph{closed gluing data for the fan picture}. Such cocycle give us for each $e:\tau\to\sigma$, an element $\ol{s}_e\in\cu{Q}_{\sigma}\otimes\bb{C}^{\times}$, which defines an automorphism $\ol{s}_e:X_{\sigma}\to X_{\sigma}$. Put
$$F_{\ol{s}}(e):=F(e)\circ\ol{s}_e.$$
Since $\ol{s}$ is a 1-cocycle, we have $\ol{s}_{e_1\circ e_2}=\ol{s}_{e_1}\circ\ol{s}_{e_2}$ and hence the $\ol{s}$-twisted limit
$$X_0(B,\msc{P},\ol{s}):=\lim_{\longrightarrow}X_{\sigma}$$
makes sense and exists in the category of algebraic spaces. In \cite{GS1}, Gross-Siebert also introduced an other more refined gluing data, called the open gluing data.

\begin{definition}
    An \emph{open gluing data for the fan picture }is a collection $s:=\{s_e\}_e$, where for $e:\tau\to\sigma$, $s_e:\Lambda_{\sigma}^*\to\bb{C}^{\times}$ is piecewise multiplicative with respective to the fan $\check{\tau}^{-1}\Sigma_{\check{\sigma}}$ such that
    \begin{enumerate}
        \item $s_{id}=1$, for $id:\sigma\to\sigma$ the identity morphism.
        \item If $e_3=e_2\circ e_1$, we have $s_{e_3}=s_{e_2}\cdot s_{e_1}$, whenever defined.
    \end{enumerate}
    An open gluing data $s$ is called \emph{trivial} if there exists $t=(t_{\sigma})_{\sigma\in\msc{P}}$, with $t_{\sigma}\in\Gamma(\sigma,\cu{Q}_{\msc{P}}\otimes\bb{C}^{\times})$ such that $s_e=t_{\tau}t_{\sigma}^{-1}|_{\tau}$, for all $e:\tau\to\sigma$. The set of all open gluing data is denoted by $Z^1(\msc{P},\cu{Q}_{\msc{P}}\otimes\bb{C}^{\times})$ and the set of all trivial open gluing data is denoted by $B^1(\msc{P},\cu{Q}_{\msc{P}}\otimes\bb{C}^{\times})$.
\end{definition}

The set of open gluing data modulo equivalence is parametrized by the group
$$ H^1(\msc{P},\cu{Q}_{\msc{P}}\otimes\bb{C}^{\times}):=\frac{Z^1(\msc{P},\cu{Q}_{\msc{P}}\otimes\bb{C}^{\times})}{B^1(\msc{P},\cu{Q}_{\msc{P}}\otimes\bb{C}^{\times})}.$$
For an open gluing data $s$, one associates the closed gluing data $\ol{s}$ as follows. For $e:\tau\to\sigma$, define $\ol{s}_e$ to be the image of $s$ under the composition
$$\Gamma(\tau,\cu{Q}_{\msc{P}}\otimes\bb{C}^{\times})\to\cu{Q}_{\tau}\otimes\bb{C}^{\times}\xrightarrow{\sim}\Gamma(W_{\tau},\cu{Q}_{\msc{P}}\otimes\bb{C}^{\times})\to\Gamma(W_e,\cu{Q}_{\msc{P}}\otimes\bb{C}^{\times}),$$
where the last map is given by restriction.
This induces the open-to-closed map
$$ H^1(\msc{P},\cu{Q}_{\msc{P}}\otimes\bb{C}^{\times})\to H^1(\cu{W},\cu{Q}_{\msc{P}}\otimes\bb{C}^{\times}),$$
which is injective by Proposition 2.32 in \cite{GS1}. An important consequence of open gluing data is that the algebraic space $X_0(B,\msc{P},\ol{s})$ can be built from some standard affine charts defined as follows. The open gluing data and the associated closed gluing data give the following commutative diagram
\begin{equation*}
        \xymatrix{
        {V_{\omega_2\to\sigma}} \ar[d]^{} \ar@{->}[r]^{s_{\omega_2\to\sigma}} & {V_{\omega_2\to\sigma}} \ar[d]^{\ol{s}_{\omega_1\to\omega_2}}
        \\ {V_{\omega_1\to\sigma}} \ar@{->}[r]^{s_{\omega_1\to\sigma}} & {V_{\omega_1\to\sigma}}
        }
\end{equation*}
The colimit of the left hand side
$$V(\sigma):=\lim_{\longrightarrow}V_{\tau\to\sigma}$$
is actually an affine scheme and hence one can construct an algebraic space $X_0(B,\msc{P},s)$ by gluing $\{V(\sigma)\}_{\sigma\in\msc{P}_{max}}$ and it was shown by using universal property of colimit that
$$X_0(B,\msc{P},s)\cong X_0(B,\msc{P},\ol{s}).$$
as algebraic spaces (Proposition 2.30 in \cite{GS1}). The fact that all cells have no self-intersections implies $X_0(B,\msc{P},s)$ is actually a scheme. Moreover, Proposition 2.32 in \cite{GS1} also showed that there is an isomorphism $X_0(B,\msc{P},s)\cong X_0(B,\msc{P},s')$ preserving toric strata if and only if $[s]=[s']$.

Assume $\msc{P}$ is simple and positive (see \cite{GS1} for their definitions). With a suitable choice of open gluing data $s$, Gross-Siebert have shown in \cite{GS1} that $X_0(B,\msc{P},s)$ carries a log structure that is log smooth off a codimension 2 locus $Z$, not containing any toric strata. Moreover, they proved in \cite{GS11} that $X_0(B,\msc{P},s)$ is smoothable to a formal family over $\Spec{\bb{C}[[t]]}$. In \cite{Ruddat_Siebert}, Ruddat and Siebert proved that this formal family is in fact an analytic family.

\section{Tropical Lagrangian multi-sections}\label{sec:trop_Lag}

We introduce the notion of tropical Lagrangian multi-sections over any integral affine manifold with singularities equipped with a polyhedral decomposition, generalizing the definition of tropical Lagrangian multi-sections in \cite{CMS_k3bundle}. We use the notion of tropical space introduced in \cite{trop_spaces}.

\begin{definition}\label{def:trop_space}
    A \emph{tropical piecewise linear space} is a pair $(X,\cu{PL}_X)$, where $\ul{X}$ is a Hausdorff paracompact topological space and $\cu{PL}_X$ is a sheaf of $\bb{R}$-valued continuous functions on $X$ such that for each $x\in X$, there is a neighborhood $U$ of $x$, an open subset $V$ of a polyhedral set in $\bb{R}^n$ for some $n$, a homeomorphism $\phi:U\to V$ and an isomorphism $\phi^{\#}:\phi^{-1}\cu{PL}_V\to\cu{PL}_U$. A \emph{tropical space} is a tropical piecewise linear space $(X,\cu{PL}_X)$ together with a choice of subsheaf $\cu{A}ff_X\subset\cu{PL}_X$ that contains the constant sheaf $\ul{\bb{R}}_X$. We simply write $X$ for $(X,\cu{PL}_X,\cu{A}ff_X)$ when there are no confusion on the tropical space structure.
\end{definition}

\begin{definition}
    Let $X$ be a tropical space. The \emph{sheaf of multi-valued piecewise linear functions} is defined to be the quotient sheaf $\cu{MPL}_X:=\cu{PL}_X/\cu{A}ff_X$.
\end{definition} 

There is a natural notion of morphisms between tropical spaces.

\begin{definition}
    Let $X,Y$ be tropical spaces. A \emph{morphism} from $f:X\to Y$ is a pair $f:=(\ul{f},f^{\#})$ where $\ul{f}:\ul{X}\to\ul{Y}$ is a continuous map between the underlying topological spaces and $f^{\#}:f^{-1}\cu{PL}_Y\to\cu{PL}_X$ is a morphism of sheaves that maps $f^{-1}\cu{A}ff_Y$ to $\cu{A}ff_X$ and $f^{-1}\ul{\bb{R}}_Y$ to $\ul{\bb{R}}_X$. A morphism of tropical spaces $f:X\to Y$ is said to be a \emph{submersion} if the induced map $f^{\#}:f^{-1}\cu{MPL}_Y\to\cu{MPL}_X$ is surjective.
\end{definition}

The following lemma is evident.

\begin{lemma}
    Let $\ul{X}$ be a topological space, $Y$ be a tropical space and $\ul{f}:\ul{X}\to \ul{Y}$ be a continuous map. The triple $(\ul{X},\ul{f}^{-1}\cu{PL}_Y,\ul{f}^{-1}\cu{A}ff_Y)$ is a tropical space.
\end{lemma}

\begin{remark}
 Any cone complex induced by a fan in some $\bb{R}$-vector space is naturally a tropical space. If $\pi:\Sigma'\to\Sigma$ is a morphism of cone complexes and $\Sigma$ is a fan, we always assume the underlying topological space $|\Sigma'|$ is equipped with the pull-back tropical structure.
\end{remark}

Given two topological spaces $L,B$, a continuous map $\pi:L\to B$ and a function $\mu:L\to\bb{Z}_{>0}$. If for any $x\in B$, the preimage set $\pi^{-1}(x)$ is finite, then we can define a function $Tr_{\pi}(\mu):B\to\bb{Z}_{>0}$ by
$$Tr_{\pi}(\mu)(x):=\sum_{x'\in\pi^{-1}(x)}\mu(x').$$
Now we can define branched covering map between tropical spaces.

\begin{definition}\label{def:branched_covering_map}
    Let $L,B$ be tropical spaces. A \emph{branched covering map} is a surjective morphism $\pi:L\to B$ and a function $\mu:L\to\bb{Z}_{>0}$, called the \emph{multiplicity map}, such that
    \begin{enumerate}
         \item For any $x\in B$, the preimage set $\pi^{-1}(x)$ is finite.
         \item For any connected open sets $W\subset B$ and connected $W'\subset \pi^{-1}(W)$, the function $Tr_{\pi|_{W'}}(\mu)$ is constant on $W$.
    \end{enumerate}
    The \emph{degree} of $\pi$ is defined to be the positive constant $Tr_{\pi}(\mu)$.
\end{definition}

Let $B$ be an integral affine manifold with singularities equipped with a polyhedral decomposition $\msc{P}$. It carries a natural tropical space structure $\cu{A}ff_B,\cu{PL}_{\msc{P}}$. See \cite{GS1}, Section 1. Unless specified, we use this tropical space structure for $(B,\msc{P})$ without further notice.

\begin{definition}\label{def:poly_decomp}
    Let $B$ be an integral affine manifold with singularities and $\msc{P}$ a polyhedral decomposition. Let $\pi:(L,\mu)\to B$ be a branched covering map between tropical spaces. A \emph{polyhedral decomposition} $\msc{P}'$ of $\pi:(L,\mu)\to B$ is a locally finite covering of $L$ by closed subsets (called cells) such that
    \begin{enumerate}
        \item If $\sigma_1',\sigma_2'\in\msc{P}'$, then $\sigma_1'\cap\sigma_2'\in\msc{P}'$.
        \item If $\sigma'\in\msc{P}'$, then $\pi(\sigma')\in\msc{P}$.
        \item For any $\sigma'\in\msc{P}'$, define the \emph{relative interior of $\sigma'$} to be
        $$\Int(\sigma'):=\sigma'\,\big\backslash\bigcup_{\tau'\in\msc{P}':\tau'\subsetneq\sigma'}\tau'.$$
        The function $\mu|_{\Int(\sigma')}$ is constant and $\pi|_{\sigma'}:\Int(\sigma')\to\Int(\pi(\sigma))$ is an isomorphism of tropical spaces with respect to the pull-back tropical structures.
    \end{enumerate}
    A cell $\sigma'$ is called \emph{ramified} if $\mu(\sigma')>1$.
\end{definition}
	
\begin{remark}
Condition (3) implies piecewise linear functions on any cell $\sigma'$ are affine functions. We use the notations $\cu{PL}_{\msc{P}'}$ and $\cu{MPL}_{\msc{P}'}$ for the sheaf of piecewise linear functions and the sheaf of multi-valued piecewise linear functions on $L$, respectively.
\end{remark}

Given $\pi:(L,\msc{P}',\mu)\to(B,\msc{P})$. For $\tau'\in\msc{P}'$ and $\tau:=\pi(\tau')$, define $W_{\tau'}$ to be the connected component of $\pi^{-1}(W_{\tau})$ that contains $\Int(\tau')$.
    
\begin{definition}\label{def:local_model}
	Let $\pi:(L,\msc{P}',\mu)\to(B,\msc{P})$ be a branched covering map of tropical spaces equipped with polyhedral decompositions. Let $x'\in L$ and $\tau'$ be the unique cell so that $x'\in \Int(\tau')$. Put $\tau:=\pi(\tau')$. A \emph{fan structure} at $x'\in L$ is a branched covering map of connected cone complexes $\pi_{\tau'}:\Sigma_{\tau'}\to\Sigma_{\tau}$ and a submersion of tropical spaces $S_{\tau'}:W_{\tau'}\to|\Sigma_{\tau'}|$ such that
	$$\pi_{\tau'}\circ S_{\tau'}=S_{\tau}\circ\pi,$$
	where $S_{\tau}:W_{\tau}\to|\Sigma_{\tau}|$ is the projection defining the fan structure at $x$. The data $\pi:(L,\msc{P}',\mu)\to(B,\msc{P})$ is called \emph{toric} if it admits a fan structure at every point.
\end{definition}

Suppose $\pi:(L,\msc{P}',\mu)\to(B,\msc{P})$ is toric and $\varphi:W_{\tau'}\to\bb{R}$ is a piecewise linear function. Then there is an affine function $f':W_{\tau'}\to\bb{R}$ and a piecewise linear function $\varphi_{\tau'}$ on $|\Sigma_{\tau'}|$ such that
$$\varphi-f'=S_{\tau'}^*\varphi_{\tau'},$$
on $W_{\tau'}$.	For $\sigma'\in\msc{P}_{max}'$ contains $\tau'$, we denote by $m_{\tau}(\sigma')\in\cu{Q}_{\tau}^*$ the slope of $\varphi_{\tau'}$ on $S_{\tau'}(\sigma')$. For $g:\tau_1\to\tau_2$, choose a path $\gamma\subset U_{\tau_1}$, which goes from a point in $\Int(\tau_1)\backslash\Delta$ to a point in $\Int(\tau_2)\backslash\Delta$. Parallel transport along $\gamma$ gives a surjection $p_g:\cu{Q}_{\tau_1}\to\cu{Q}_{\tau_2}$. Given $\varphi\in  H^0(\cu{W}',\cu{MPL}_{\msc{P}'})$ and representatives $\{\varphi_{\tau'}\}$, for $\tau_1'\subset\tau_2'$, there exists an affine function $f_{\tau_1'\tau_2'}:W_{\tau_1'}\cap W_{\tau_2'}\to\bb{R}$ such that
$$S_{\tau_2'}^*\varphi_{\tau_2'}=S_{\tau_1'}^*\varphi_{\tau_1'}+f_{\tau_1'\tau_2'},$$
whenever defined. Therefore, via the inclusion $p_g^*:\cu{Q}_{\tau_2}^*\to\cu{Q}_{\tau_1}^*$, for any $\sigma'\supset\tau_2'$,
$$p_g^*m_{\tau_2}(\sigma')=m_{\tau_1}(\sigma')+m_{\tau_1'\tau_2'},$$
for some $m_{\tau_1'\tau_2'}\in\cu{Q}_{\tau_1}^*$ only depends on $\tau_1',\tau_2'$. We simply write
$$m_{\tau_2}(\sigma')=m_{\tau_1}(\sigma')+m_{\tau_1'\tau_2'}$$
if there is no confusion.

Now we can define the main object that we are going to study in this paper.

\begin{definition}\label{def:trop_lag}
	Let $B$ be an integral affine manifold with singularities and $\msc{P}$ a polyhedral decomposition. A \emph{tropical Lagrangian multi-section} $\bb{L}$ over $(B,\msc{P})$ is a toric branched covering of tropical spaces $\pi:(L,\msc{P}',\mu)\to(B,\msc{P})$ equipped with polyhedral decomposition, together with a global section $\varphi\in  H^0(L,\cu{MPL}_{\msc{P}'})$.
\end{definition}

There is a special type of morphisms between tropical Lagrangian multi-sections over the same base $(B,\msc{P})$.

\begin{definition}\label{def:covering_morphism}
    A \emph{covering morphism} of tropical Lagrangian multi-sections $f:\bb{L}_1\to\bb{L}_2$ over $(B,\msc{P})$ is a surjective morphism of tropical spaces $f:L_1\to L_2$, mapping cells in $\msc{P}_1'$ isomorphically onto cells in $\msc{P}_2'$ such that $\pi_1=\pi_2\circ f$, $Tr_f(\mu_1)=\mu_2$ and $\varphi_1=f^*\varphi_2$.
\end{definition}

\section{From tropical Lagrangian multi-sections to locally free sheaves}\label{sec:constructing_E_0}

Let $\bb{L}$ be a tropical Lagrangian multi-section over $(B,\msc{P})$ of degree $r$ and $ s$ an open gluing data. By thinking $\bb{L}$ as a Lagrangian multi-section of a Lagrangian torus fibration over $B$, the SYZ philosophy suggests the mirror of $\bb{L}$ should be a holomorphic vector bundle, whose rank is same as the degree of the covering $\bb{L}\to B$. Therefore, in this section, we would like to construct a rank $r$ locally free sheaf on $X_0(B,\msc{P},s)$. However, as mentioned in the introduction, one shouldn't expect $\bb{L}$ itself can determine a locally free sheaf due to its discrete nature. We need some extra continuous data in analogous to the linear algebra data defined in \cite{Kaneyama_classification}.

To begin, let $\tau\in\msc{P}$ and $\tau'\in\msc{P}'$ be a lift, we would like to construct a rank $\mu(\tau')$ locally free sheaf on the strata $X_{\tau}$. Let $V_{\tau\to\sigma}\subset X_{\tau}$ be the affine chart corresponds to the cone $K_{\tau\to\sigma}:=\bb{R}_{\geq 0}\cdot S_{\tau}(\sigma)\in\Sigma_{\tau}$. Define
$$\cu{E}_{\sigma}(\tau'):=\cu{O}_{V_{\tau\to\sigma}}^{\oplus\mu(\tau')}.$$
For $\sigma\in\msc{P}_{max}$ contains $\tau$, $\mu(\tau')$ equals to the number (count with multiplicity) of lifts of $\sigma$ that contain $\tau'$. We then obtain a frame $\{1_{\sigma^{(\alpha)}}(\tau)\}_{\sigma^{(\alpha)}\supset\tau'}$ for $\cu{E}_{\sigma}(\tau')$, parametrized by lifts of $\sigma$ contains $\tau'$, counting with multiplicity. To define transition maps, we use the function $\varphi$. By the toric assumption, there is a connected cone complex $\Sigma_{\tau'}$ over $\Sigma_{\tau}$ and a piecewise linear function $\varphi_{\tau'}$ such that $S_{\tau'}^*\varphi_{\tau'}$ represents $\varphi|_{W_{\tau'}}$. Let $\sigma'\in\msc{P}_{max}'$ be a lift of $\sigma$ contains $\tau'$ and $m_{\tau}(\sigma')\in\cu{Q}_{\tau}^*$ be the slope of $\varphi_{\tau'}$ on the cone $S_{\tau'}(\sigma')$. For $\sigma_1,\sigma_2\supset\tau$, define $G_{\sigma_1\sigma_2}(\tau'):\cu{E}_{\sigma_1}(\tau')|_{V_{\tau\to\sigma_1\cap\sigma_2}}\to\cu{E}_{\sigma_2}(\tau')|_{V_{\tau\to\sigma_1\cap\sigma_2}}$ by
$$G_{\sigma_1\sigma_2}(\tau'):1_{\sigma_1^{(\alpha)}}(\tau')\mapsto\sum_{\beta:\sigma_2^{(\beta)}\supset\tau'}g_{\sigma_1^{(\alpha)}\sigma_2^{(\beta)}}(\tau')z^{m_{\tau}(\sigma_1^{(\alpha)})-m_{\tau}(\sigma_2^{(\beta)})}1_{\sigma_2^{(\beta)}}(\tau'),$$
where $\{1_{\sigma^{(\alpha)}}(\tau')\}$ is a frame of $\cu{E}_{\sigma}(\tau')$. Put
$$G_{\sigma_1\sigma_2}(\tau):=\sum_{\tau':\pi(\tau')=\tau}G_{\sigma_1\sigma_2}(\tau').$$
The coefficient of each monomial entry of $G_{\sigma_1\sigma_2}(\tau)$ will be denoted by $g_{\sigma_1^{(\alpha)}\sigma_2^{(\beta)}}(\tau)\in\bb{C}$. We require them to satisfy the following
    
\begin{definition}\label{def:comp_data}
    Let $\tau'\in\msc{P}'$. A \emph{$\tau'$-Kaneyama data} is a collection of invertible matrices $${\bf{g}}(\tau'):=\{(g_{\sigma_1^{(\alpha)}\sigma_2^{(\beta)}}(\tau'))\}_{\tau'\subset\sigma_1^{(\alpha)},\sigma_2^{(\beta)}\in\msc{P}'_{max}}\subset GL(\mu(\tau'),\bb{C})$$
    such that
    \begin{enumerate}
        \item [(G1)] $g_{\sigma^{(\alpha)}\sigma^{(\beta)}}(\tau')=Id^{(\alpha\beta)}$, for all $\sigma\in\msc{P}_{max}$.
        \item [(G2)] $g_{\sigma_1^{(\alpha)}\sigma_2^{(\beta)}}(\tau')=0$ if  $m_{\tau}(\sigma_1^{(\alpha)})-m_{\tau}(\sigma_2^{(\beta)})\notin K_{\tau\to\sigma_1\cap\sigma_2}^{\vee}\cap\cu{Q}_{\tau}^*$.
        \item [(G3)] For any $\sigma_1,\sigma_2,\sigma_3\in\msc{P}_{max}$, we have
        \begin{equation*}
            \sum_{\beta:\sigma_2^{(\beta)}\supset\tau'}g_{\sigma_1^{(\alpha)}\sigma_2^{(\beta)}}(\tau')g_{\sigma_2^{(\beta)}\sigma_3^{(\gamma)}}(\tau')=g_{\sigma_1^{(\alpha)}\sigma_3^{(\gamma)}}(\tau'),
        \end{equation*}
        for all $\sigma_1^{(\alpha)},\sigma_3^{(\gamma)}\supset\tau'$.
        \end{enumerate}
    A collection of Kaneyama data ${\bf{g}}:=\{{\bf{g}}(\tau')\}_{\tau'\in\msc{P}'}$ is said to be \emph{compatible} if ${\bf{g}}(\tau')$ is $\tau'$-compatible for all $\tau'\in\msc{P}'$ and for each $g:\tau_1\to\tau_2$, there exist a collection of $r\times r$ matrices matrix $${\bf{h}}(g):=\{(h_{\sigma^{(\alpha)}\sigma^{(\beta)}}(g))_{\alpha,\beta}\}_{\sigma\in\msc{P}_{max}}\subset GL(r,\bb{C})$$
    such that
    \begin{enumerate}
        \item [(H1)] For any $\sigma\in\msc{P}_{max}$ contains $\tau_2$, we have $h_{\sigma^{(\alpha)}\sigma^{(\beta)}}(g)\neq 0$ only if $\sigma^{(\alpha)},\sigma^{(\beta)}$ contains a common lift of $\tau_1$ and $m_{\tau_1}(\sigma^{(\alpha)})-m_{\tau_1}(\sigma^{(\beta)})\in K_{\tau_1\to\sigma}^{\vee}\cap\cu{Q}_{\tau_2}^*$.
        \item [(H2)] For any $\sigma_1,\sigma_2\in\msc{P}_{max}$ contain $\tau_2$,
        $$\sum_{\beta=1}^rh_{\sigma_1^{(\alpha)}\sigma_1^{(\beta)}}(g)g_{\sigma_1^{(\beta)}\sigma_2^{(\gamma)}}(\tau_1)=\sum_{\beta=1}^rg_{\sigma_1^{(\alpha)}\sigma_2^{(\beta)}}(\tau_2)h_{\sigma_2^{(\beta)}\sigma_2^{(\gamma)}}(g),$$
        whenever $m_{\tau_1}(\sigma_1^{(\alpha)})-m_{\tau_1}(\sigma_2^{(\gamma)})\in K_{\tau_1\to\sigma_1\cap\sigma_2}^{\vee}\cap\cu{Q}_{\tau_2}^*$.
        \item [(H3)] For $g_1:\tau_1\to\tau_2,g_2:\tau_2\to\tau_3$ and $g_3:=g_2\circ g_1$, we have
        $$\sum_{\beta=1}^rh_{\sigma^{(\alpha)}\sigma^{(\beta)}}(g_2)h_{\sigma^{(\beta)}\sigma^{(\gamma)}}(g_1)=h_{\sigma^{(\alpha)}\sigma^{(\gamma)}}(g_3),$$
        whenever $\sigma\supset\tau_3$ and $m_{\tau_1}(\sigma^{(\alpha)})-m_{\tau_1}(\sigma^{(\gamma)})\in K_{\tau_1\to\sigma}^{\vee}\cap\cu{Q}_{\tau_3}^*$.
    \end{enumerate}
\end{definition}

\begin{remark}
 Being invertible and the cocycle condition (G3) are independent of the choice of the ordering $\sigma^{(1)},\dots,\sigma^{(r)}$, so Definition \ref{def:comp_data} only depends on $\bb{L}$.
\end{remark}

\begin{remark}
Conditions (G1)-(G3) are generalization of the linear algebra data given in \cite{Kaneyama_classification} to affine manifold with singularities. Given a tropical Lagrangian multi-section with degree $\geq 2$, a Kaneyama data may not exist, even on a single toric piece (see \cite{Suen_trop_lag}, Example 5.1). Therefore, one may ask for the abundance of such data. We will prove in Theorem \ref{thm:many_examples} that, at least in the case of Calabi-Yau hypersurfaces, such data can be obtained from restricting toric vector bundles on the ambient toric variety to its boundary divisor.
\end{remark}

Condition (G2) implies entries of $G_{\sigma_1\sigma_2}(\tau')$ are regular functions. Condition (G1) and the cocycle condition (G3) immediately implies the existence of a rank $\mu(\tau')$ locally free sheaf $\cu{E}({\bf{g}}(\tau'))$ on the closed toric strata $X_{\tau}$. Define
$$\cu{E}({\bf{g}}(\tau)):=\bigoplus_{\tau':\pi(\tau')=\tau}\cu{E}({\bf{g}}(\tau')),$$
which is a rank $r$ locally sheaf on $X_{\tau}$.

\begin{remark}\label{rem:equiv_str}
 The local representative $\varphi_{\tau'}$ of $\varphi$ determines a $\cu{Q}_{\tau}\otimes\bb{C}^{\times}$-action on $\cu{E}({\bf{g}}(\tau'))$. Namely,
 $$\lambda\cdot 1_{\sigma^{(\alpha)}}(\tau'):=\lambda^{m_{\tau}(\sigma^{(\alpha)})}1_{\sigma^{(\alpha)}}(\tau'),$$
 for all $\lambda\in\cu{Q}_{\tau}\otimes\bb{C}^{\times}$. One can easily check that this action is compatible with the transition maps. Hence $\cu{E}({\bf{g}}(\tau'))$ carries a structure of toric vector bundle over $X_{\tau}$. The existence of equivariant structure will be important when we perform the reverse construction in Section \ref{sec:E_to_L}.
\end{remark}
    
We would like to glue $\{\cu{E}({\bf{g}}(\tau))\}_{\tau\in\msc{P}}$ together. The idea is to embed $\cu{E}(\tau_2')$ to $\cu{E}(\tau_1')$ when $\tau_1'\subset\tau_2'$. To do this, we first use the data $\{(h_{\sigma^{(\alpha)}\sigma^{(\beta)}}(g))\}$ to construct an isomorphism $H(g):\cu{E}({\bf{g}}(\tau_2))\to F(g)^*\cu{E}({\bf{g}}(\tau_1))$. Define
$$H_{\sigma}(g):1_{\sigma^{(\alpha)}}(\tau_2)\mapsto\sum_{\beta=1}^rh_{\sigma^{(\alpha)}\sigma^{(\beta)}}(g)z^{m_{\tau_1}(\sigma^{(\alpha)})-m_{\tau_1}(\sigma^{(\beta)})}|_{V_{\tau_2\to\sigma}}F(g)^*1_{\sigma^{(\beta)}}(\tau_1).$$
By Condition (H1), the entries of $H_{\sigma}(g)$ are regular functions on $V_{\tau_2\to\sigma}$. Given maximal $\sigma_1,\sigma_2\supset\tau_2$, the composition $F(g)^*G_{\sigma_1\sigma_2}(\tau_1)\circ H_{\sigma_1}(g)|_{V_{\tau_2\to\tau}}$ is given by
$$1_{\sigma_1^{(\alpha)}}(\tau_1)\mapsto\sum_{\beta,\gamma=1}^rh_{\sigma_1^{(\alpha)}\sigma_1^{(\beta)}}(g)g_{\sigma_1\sigma_2}^{(\beta\gamma)}(\tau_1)z^{m_{\tau_1}(\sigma_1^{(\alpha)})-m_{\tau_1}(\sigma_2^{(\gamma)})}|_{V_{\tau_2\to\tau}}F(g)^*1_{\sigma_2^{(\gamma)}}(\tau_1).$$
On the other hand, the composition $H_{\sigma_2}(g)|_{V_{\tau_2\to\tau}}\circ G_{\sigma_1\sigma_2}(\tau_2)$ is given by
$$1_{\sigma_1^{(\alpha)}}(\tau_2)\mapsto\sum_{\beta,\gamma=1}^rg_{\sigma_1^{(\alpha)}\sigma_2^{(\beta)}}(\tau_2)h_{\sigma_2^{(\beta)}\sigma_2^{(\gamma)}}(g)z^{m_{\tau_2}(\sigma_1^{(\alpha)})-m_{\tau_2}(\sigma_2^{(\beta)})+m_{\tau_1}(\sigma_2^{(\beta)})-m_{\tau_1}(\sigma_2^{(\gamma)})}|_{V_{\tau_2\to\tau}}F(g)^*1_{\sigma_2^{(\gamma)}}(\tau_1).$$
Now, we introduce a frequently used trick, called the \emph{slope cancellation trick}. By definition of $G_{\sigma_1\sigma_2}(\tau_2)$, the constant $g_{\sigma_1^{(\alpha)}\sigma_2^{(\beta)}}(\tau_2)$ is non-zero only if $\sigma_1^{(\alpha)},\sigma_2^{(\beta)}$ contain a common lift of $\tau_2$, say $\tau_2'$, so in particular, they contains $\tau_1'\subset\tau_2'$. On the other hand, by the construction of $H_{\sigma}(g)$, the constant $h_{\sigma_2^{(\beta)}\sigma_2^{(\gamma)}}(g)$ is non-zero only if $\sigma_2^{(\beta)},\sigma_2^{(\gamma)}$ contains a common lift of $\tau_1$ and it must be $\tau_1'$ as $\sigma_2^{(\beta)}\supset\tau_1'$. As a whole, we conclude that $\sigma_1^{(\alpha)},\sigma_2^{(\beta)},\sigma_2^{(\gamma)}$ all contain the lift $\tau_1'$ of $\tau_1$. Moreover, via the inclusion $p_g^*:\cu{Q}_{\tau_2}^*\to\cu{Q}_{\tau_1}^*$, the piecewise linear function
$$f:=p_g^*m_{\tau_2'}(\sigma')-m_{\tau_1'}(\sigma')$$
is independent of $\sigma'$ as long as $\tau_1'\subset\tau_2'\subset\sigma'$, which means $f$ is actually an affine function. This implies
\begin{align*}
    m_{\tau_2'}(\sigma_1^{(\alpha)})-m_{\tau_2'}(\sigma_2^{(\beta)})+m_{\tau_1'}(\sigma_2^{(\beta)})-m_{\tau_1'}(\sigma_2^{(\gamma)})=&\,m_{\tau_2'}(\sigma_1^{(\alpha)})-f-m_{\tau_1'}(\sigma_2^{(\gamma)})\\
    =&\,m_{\tau_1'}(\sigma_1^{(\alpha)})-m_{\tau_1'}(\sigma_2^{(\gamma)}).
\end{align*}
It is worth mentioning that we are not allowed to absorb $f$ by $m_{\tau_1'}(\sigma_2^{(\gamma)})$ as $\sigma_2^{(\gamma)}$ may not contain $\tau_2'$.

By using the slope cancellation trick and Condition (H2), it is easy to see that
$$F(g)^*G_{\sigma_1\sigma_2}(\tau_1)\circ H_{\sigma_1}(g)|_{V_{\tau_2\to\tau}}=H_{\sigma_2}(g)|_{V_{\tau_2\to\tau}}\circ G_{\sigma_1\sigma_2}(\tau_2).$$
Moreover, we have
\begin{align*}
\det(H_{\sigma}(g))=&\,\det(h_{\sigma^{(\alpha)}\sigma^{(\beta)}}(g))z^{\sum_{\alpha=1}^rm_{\tau_1}(\sigma^{(\alpha)})-\sum_{\alpha=1}^rm_{\tau_1}(\sigma^{(\alpha)})}\\
=&\,\det(h_{\sigma^{(\alpha)}\sigma^{(\beta)}}(g))\in\bb{C}^{\times}.
\end{align*}
Hence $H(g):\cu{E}({\bf{g}}(\tau_2))\to F(g)^*\cu{E}({\bf{g}}(\tau_1))$ defines an isomorphism.

Let $\ol{s}$ be the closed gluing data associated to the open gluing data $s$. We now define
$$H_{\ol{s}}(g):\cu{E}({\bf{g}}(\tau_2))\to F_{\ol{s}}(g)^*\cu{E}({\bf{g}}(\tau_1))$$
to be the composition
$$\cu{E}({\bf{g}}(\tau_2))\to\ol{s}_g^*\cu{E}({\bf{g}}(\tau_2))\xrightarrow{\ol{s}_g^*H(g)}F_{\ol{s}}(g)^*\cu{E}({\bf{g}}(\tau_1))$$
where the first isomorphism is prescribed by the chosen equivariant structure on $\cu{E}({\bf{g}}(\tau_2))$, which depends on the choice of local representatives $\{\varphi_{\tau_2'}\}$. Explicitly, it is given by
$$H_{\ol{s}}(g):1_{\sigma^{(\alpha)}}(\tau_2)\mapsto\sum_{\beta=1}^rh_{\sigma^{(\alpha)}\sigma^{(\beta)}}(g)\frac{\ol{s}_g(m_{\tau_1}(\sigma^{(\alpha)})-m_{\tau_1}(\sigma^{(\beta)}))}{\ol{s}_g(m_{\tau_2}(\sigma^{(\alpha)}))}z^{m_{\tau_1}(\sigma^{(\alpha)})-m_{\tau_1}(\sigma^{(\beta)})}|_{V_{\tau_2\to\sigma}}F_{\ol{s}}(g)^*1_{\sigma^{(\beta)}}(\tau_1).$$
To obtain a consistent gluing, we need the following cocycle condition
\begin{equation}\label{eqn:cocycle_H}
    H_{\ol{s}}(g_3)^{-1}\circ F_{\ol{s}}(g_2)^*H_{\ol{s}}(g_1)\circ H_{\ol{s}}(g_2)=Id_{\cu{E}({\bf{g}}(\tau_3))},
\end{equation}
for all $g_1:\tau_1\to\tau_2,g_2:\tau_2\to\tau_3$ and $g_3:=g_2\circ g_1$. We only need to check this for any triple $\tau_1'\subset\tau_2'\subset\tau_3'$. Consider the composition
\begin{equation}\label{eqn:composition_H}
    H_{\ol{s}}(g_3)^{-1}\circ F_{\ol{s}}(g_2)^*H_{\ol{s}}(g_1)\circ H_{\ol{s}}(g_2).
\end{equation}
First note that the monomial part of a summand of (\ref{eqn:composition_H}) has exponent
$$(m_{\tau_2}(\sigma^{(\alpha)})-m_{\tau_2}(\sigma^{(\beta)}))+(m_{\tau_1}(\sigma^{(\beta)})-m_{\tau_1}(\sigma^{(\gamma)}))+(m_{\tau_1}(\sigma^{(\gamma)})-m_{\tau_1}(\sigma^{(\delta)})),$$
with each bracketed term lies in $\cu{Q}_{\tau_3}^*$. The corresponding summand is non-zero only if $\sigma^{(\alpha)},\sigma^{(\beta)},\sigma^{(\gamma)},\sigma^{(\delta)}$ contain a common lift of $\tau_1$. Using the slope cancellation trick, the exponent reduces to
$$m_{\tau_1}(\sigma^{(\alpha)})-m_{\tau_1}(\sigma^{(\delta)})\in\cu{Q}_{\tau_3}^*.$$
Now, the coefficient of the $(\alpha,\delta)$-entry of (\ref{eqn:composition_H}) is given by
$$\sum_{\beta,\gamma=1}^rh_{\sigma^{(\alpha)}\sigma^{(\beta)}}(g_2)h_{\sigma^{(\beta)}\sigma^{(\gamma)}}(g_1)h_{\sigma^{(\gamma)}\sigma^{(\delta)}}^{-1}(g_3)\ol{s}_{\tau_1\tau_2\tau_3}^{(\alpha\beta\gamma\delta)}(\sigma),$$
where $\ol{s}_{\tau_1\tau_2\tau_3}^{(\alpha\beta\gamma\delta)}(\sigma)$ is the product of the following factors
\begin{align*}
    &\ol{s}_{g_2}(m_{\tau_2}(\sigma^{(\alpha)})-m_{\tau_2}(\sigma^{(\beta)}))\ol{s}_{g_2}(m_{\tau_3}(\sigma^{(\alpha)}))^{-1},\\
    &\ol{s}_{g_3}(m_{\tau_1}(\sigma^{(\beta)})-m_{\tau_1}(\sigma^{(\gamma)}))\ol{s}_{g_1}(m_{\tau_2}(\sigma^{(\beta)}))^{-1},\\
    &\ol{s}_{g_3}(m_{\tau_1}(\sigma^{(\gamma)})-m_{\tau_1}(\sigma^{(\delta)}))\ol{s}_{g_3}(m_{\tau_3}(\sigma^{(\delta)})).
\end{align*}
Using the slope cancellation trick again, all the slope difference becomes $m_{\tau_1}(\sigma^{(\eta)})-m_{\tau_1}(\sigma^{(\xi)})$. Then one can easily show that
$$\ol{s}_{\tau_1\tau_2\tau_3}^{(\alpha\beta\gamma\delta)}(\sigma)=\ol{s}_{\tau_1\tau_2\tau_3}(\sigma^{(\alpha)})^{-1},$$
where
$$\ol{s}_{\tau_1\tau_2\tau_3}(\sigma^{(\alpha)}):=\ol{s}_{g_1}(m_{\tau_2}(\sigma^{(\alpha)}))\ol{s}_{g_2}(m_{\tau_3}(\sigma^{(\alpha)}))\ol{s}_{g_3}(m_{\tau_3}(\sigma^{(\alpha)}))^{-1}.$$
Using Condition (H3), the composition (\ref{eqn:composition_H}) can be simplified to
$$1_{\sigma^{(\alpha)}}(\tau_3)\mapsto \ol{s}_{\tau_1\tau_2\tau_3}(\sigma^{(\alpha)})^{-1}1_{\sigma^{(\alpha)}}(\tau_3).$$

\begin{lemma}\label{lem:obs_class}
    The 2-cocycle $\ol{s}_{\tau_1\tau_2\tau_3}(\sigma^{(\alpha)})$ only depends on the lifts $\tau_1',\tau_2',\tau_3'$ so that $\tau_1'\subset\tau_2'\subset\tau_3'\subset\sigma^{(\alpha)}$. It is then closed with respective to the C\v{e}ch differential $\check{\delta}$ on $\check{C}^2(\cu{W}',\bb{C}^{\times})$ and its cohomology class is independent of the local representatives of $\varphi$. 
\end{lemma}
\begin{proof}
     Let $\sigma_1,\sigma_2\in\msc{P}_{max}$ such that $\sigma_1,\sigma_2\supset\tau_3$. We first prove the special case that $\sigma_1^{(\alpha)},\sigma_2^{(\beta)}$ contain the common lift $\tau_3'$ so that
     $$m_{\tau_3}(\sigma_1^{(\alpha)})-m_{\tau_3}(\sigma_2^{(\beta)})\in K_{\tau_3\to\sigma_1\cap\sigma_2}^{\vee}\cap\cu{Q}_{\tau_3}^*.$$
     In this case, via the inclusion $\cu{Q}_{\tau_3}^*\subset\cu{Q}_{\tau_2}^*\subset\cu{Q}_{\tau_1}^*$, we have
     $$m_{\tau_1}(\sigma_1^{(\alpha)})-m_{\tau_1}(\sigma_2^{(\beta)})=m_{\tau_2}(\sigma_1^{(\alpha)})-m_{\tau_2}(\sigma_2^{(\beta)})=m_{\tau_3}(\sigma_1^{(\alpha)})-m_{\tau_3}(\sigma_2^{(\beta)}).$$
     Hence the cocycle condition of $ s$ implies
     $$\ol{s}_{\tau_1\tau_2\tau_3}(\sigma_1^{(\alpha)})=\ol{s}_{\tau_1\tau_2\tau_3}(\sigma_2^{(\beta)}).$$
     For general pair of $\sigma_1^{(\alpha)},\sigma_2^{(\beta)}$ that contains $\tau_3'$, choose a sequence of maximal cells
     $$\sigma_1^{(\alpha)}:=\sigma_{i_1}^{(\alpha_1)},\sigma_{i_2}^{(\alpha_2)},\dots,\sigma_{i_k}^{(\alpha_k)}:=\sigma_2^{(\beta)}.$$
     such that $\sigma_{i_j}^{(\alpha_j)}\supset\tau_3'$ for all $j=1,\dots,k$ and $\pi(\sigma_{i_j}^{(\alpha_j)}\cap\sigma_{i_{j+1}}^{(\alpha_{j+1})})=\sigma_{i_j}\cap\sigma_{i_{j+1}}$ for all $j=1,\dots,k-1$. Then continuity of $\varphi_{\tau_3'}$ implies
     $$m_{\tau_3}(\sigma_{i_j}^{(\alpha_j)})|_{K_{\tau_3\to\sigma_{i_j}\cap\sigma_{i_{j+1}}}}=m_{\tau_3}(\sigma_{i_{j+1}}^{(\alpha_{j+1})})|_{K_{\tau_3\to\sigma_{i_j}\cap\sigma_{i_{j+1}}}}$$
     for all $j=1,\dots,k-1$. In particular,
     $$m_{\tau_3}(\sigma_{i_j}^{(\alpha_j)})-m_{\tau_3}(\sigma_{i_{j+1}}^{(\alpha_{j+1})})\in K_{\tau_3\to\sigma_{i_j}\cap\sigma_{i_{j+1}}}^{\vee}\cap\cu{Q}_{\tau_3}^*,$$
     for all $j=1,\dots,k-1$. By the special case, we have
     $$\ol{s}_{\tau_1\tau_2\tau_3}(\sigma_1^{(\alpha)})=\ol{s}_{\tau_1\tau_2\tau_3}(\sigma_{i_2}^{(\alpha_2)})=\cdots=\ol{s}_{\tau_1\tau_2\tau_3}(\sigma_2^{(\beta)}).$$
     This proves the first part of the lemma. For the second part, it is obvious that $\ol{s}_{\tau_1\tau_2\tau_3}(\sigma^{(\alpha)})$ is $\check{\delta}$-closed. To show that its cohomology class is independent of the local representatives, note that any choice of another local representative of $\varphi$ differ from $\varphi_{\tau_i'}$ by a local affine function $f_{\tau_i'}:W_{\tau_i}\to\bb{R}$. Then
     $$\ol{s}_{\tau_1\tau_2\tau_3}^{old}(\sigma_1^{(\alpha)})=\ol{s}_{g_1}(f_{\tau_2'})\ol{s}_{g_2}(f_{\tau_3'})\ol{s}_{g_3}(f_{\tau_3'})^{-1}\ol{s}_{\tau_1\tau_2\tau_3}^{new}(\sigma_1^{(\alpha)}),$$
     which means $\ol{s}_{\tau_1\tau_2\tau_3}^{old}(\sigma_1^{(\alpha)})$ and $\ol{s}_{\tau_1\tau_2\tau_3}^{new}(\sigma_1^{(\alpha)})$ define the same cohomology class.
\end{proof}
    
Denote the cohomology class obtained in Lemma \ref{lem:obs_class} by $[\ol{s}_{\tau_1'\tau_2'\tau_3'}]\in  H^2(\cu{W}',\bb{C}^{\times})$. We define $o_{\bb{L}}: H^1(\cu{W},\bb{C}^{\times})\to  H^2(\cu{W}',\bb{C}^{\times})$ by
$$o_{\bb{L}}:[\ol{s}]\mapsto[\ol{s}_{\tau_1'\tau_2'\tau_3'}].$$
This map is well-defined because if $\ol{s}_{\tau_1\to\tau_2}=t_{\tau_1}^{-1}t_{\tau_2}|_{\tau_1}$, we have
\begin{align*}
    \ol{s}_{\tau_1\tau_2\tau_3}(\sigma^{(\alpha)})=&\,\frac{t_{\tau_2}(m_{\tau_2}(\sigma^{(\alpha)}))}{t_{\tau_1}(m_{\tau_2}(\sigma^{(\alpha)}))}
    \frac{t_{\tau_3}(m_{\tau_3}(\sigma^{(\alpha)}))}{t_{\tau_2}(m_{\tau_3}(\sigma^{(\alpha)}))}
    \frac{t_{\tau_1}(m_{\tau_3}(\sigma^{(\alpha)}))}{t_{\tau_3}(m_{\tau_3}(\sigma^{(\alpha)}))}
    \\=&\,\frac{t_{\tau_1}(m_{\tau_1}(\sigma^{(\alpha)}))}{t_{\tau_1}(m_{\tau_2}(\sigma^{(\alpha)}))}
    \frac{t_{\tau_2}(m_{\tau_2}(\sigma^{(\alpha)}))}{t_{\tau_2}(m_{\tau_3}(\sigma^{(\alpha)}))}
    \frac{t_{\tau_1}(m_{\tau_3}(\sigma^{(\alpha)}))}{t_{\tau_1}(m_{\tau_1}(\sigma^{(\alpha)}))}
    \\=&\,t_{\tau_1}(m_{\tau_1'\tau_2'})t_{\tau_2}(m_{\tau_2'\tau_3'})t_{\tau_1'}(m_{\tau_1'\tau_3'})^{-1}.
\end{align*}
It is clear that $o_{\bb{L}}$ is a group homomorphism. Since $W_{\tau_0'}\cap\cdots\cap W_{\tau_p'}$ are contractible for all $p\geq 0$, it follows that $\{W_{\tau'}\}_{\tau'\in\msc{P}'}$ is an acyclic cover for $\bb{C}^{\times}$. It was also shown is \cite{GS1} that $\{W_{\tau}\}_{\tau\in\msc{P}}$ is an acyclic cover for $\cu{Q}_{\msc{P}}\otimes\bb{C}^{\times}$. Therefore, we can simply write  $o_{\bb{L}}: H^1(B,\cu{Q}_{\msc{P}}\otimes\bb{C}^{\times})\to H^2(L,\bb{C}^{\times})$.

\begin{theorem}\label{thm:close}
    Suppose $\ol{s}$ is the associated closed gluing data of an open gluing data $s$. The locally free sheaves $\{\cu{E}({\bf{g}}(\tau))\}_{\tau\in\msc{P}}$ can be glued to a rank $r$ locally free sheaf on the scheme $X_0(B,\msc{P},s)$ via the data $({\bf{g}},{\bf{h}})$ if and only if $o_{\bb{L}}([\ol{s}])=1$.
\end{theorem}
\begin{proof}
 If $\{\cu{E}({\bf{g}}(\tau))\}_{\tau\in\msc{P}}$ can be glued, then it is necessary that $s_{\tau_1'\tau_2'\tau_3'}=1$. Hence its cohomology class equals to 1 too. Conversely, suppose $o_{\bb{L}}([ s])=1$. Then there exists a collection ${\bf{k}}_s:=\{k_{\tau_1'\tau_2'}\}_{\tau_1'\subset\tau_2'}$ such that
 $$\ol{s}_{\tau_1'\tau_2'\tau_3'}=k_{\tau_1'\tau_2'}k_{\tau_2'\tau_3'}k_{\tau_1'\tau_3'}^{-1}.$$
 We modify $H_{\ol{s}}(g)$ to a map $\til{H}_{\ol{s}}(g)$, given by
 $$1_{\tau_2'}^{(\alpha)}(\sigma)\mapsto\sum_{\beta=1}^rk_{\tau_1^{(\alpha)}\tau_2^{(\alpha)}}^{-1}\frac{\ol{s}_g(m_{\tau_1'}(\sigma^{(\alpha)})-m_{\tau_1'}(\sigma^{(\beta)}))}{\ol{s}_{g_1}(m_{\tau_2}(\sigma^{(\alpha)}))}h_{\sigma^{(\alpha)}\sigma^{(\beta)}}(g)z^{m_{\tau_1'}(\sigma^{(\alpha)})-m_{\tau_1'}(\sigma^{(\beta)})}F_{\ol{s}}(g)^*1_{\tau_1'}^{(\beta)}(\sigma)$$
 where, up to reordering, $\tau_1^{(\alpha)},\tau_2^{(\alpha)}$ are determined by $\tau_1^{(\alpha)}\subset\tau_2^{(\alpha)}\subset\sigma^{(\alpha)}$. Then it is easy to see that
 \begin{equation}\label{eqn:cocycle_H}
     F_{\ol{s}}(g_2)^*\til{H}_{\ol{s}}(g_1)\circ\til{H}_{\ol{s}}(g_2)=\til{H}_{\ol{s}}(g_3).
 \end{equation}
 We can then define the colimit $\cu{E}_0:=\displaystyle{\lim_{\longrightarrow}}\,\cu{E}({\bf{g}}(\tau))$ with respective to $\{\til{H}_{\tau_1\tau_2}(s)\}$.
     
 It remains to prove locally freeness. To do this, we describe $\cu{E}_0$ on open subsets of $X_0(B,\msc{P}, s)$. For each maximal $\sigma\in\msc{P}_{max}$, let
 $$V_{\ol{s}}(\sigma):=\lim_{\substack{\longrightarrow\\\tau\subset\sigma}}V_{\tau\to\sigma}.$$
 Recall that $V_{\ol{s}}(\sigma)\cong V(\sigma)$, which is an affine scheme. Denote by $i_{\sigma}:V_{\ol{s}}(\sigma)\hookrightarrow X_0(B,\msc{P},\ol{s})$ the inclusion, given by embedding an affine strata $V_{\tau\to\sigma}$ to the closed strata $X_{\tau}\subset X_0(B,\msc{P},s)$. Then
 $$X_0(B,\msc{P},s)=\bigcup_{\sigma\in\msc{P}_{max}}i_{\sigma}(V_{\ol{s}}(\sigma)).$$
 Let $v_1,v_2$ be two vertices of $\sigma$. By definition of the limit, $\xi_1(x_1)\in\cu{E}_{\sigma}(v_1)_{x_1}$ is identified with $\xi_2(x_2)\in\cu{E}_{\sigma}(v_2)_{x_2}$ if and only if there exists $\tau\subset\sigma$, $g_1:v_1\to\tau,g_2:v_2\to\tau$ and $x\in V_{\tau\to\sigma}$ with
 $$F_{\ol{s}}(g_i)(x)=x_i$$
 and there exists $\eta(x)\in\cu{E}_{\sigma}(\tau)_x$ such that
 $$\til{H}_{\ol{s}}(g_i):\eta(x)\mapsto\xi_i(x_i),$$
 for $i=1,2$. This is an equivalence relation due to the cocycle condition of $\{\til{H}_{\ol{s}}(g)\}_g$. For $e:\tau\to\sigma$. Define
 $$\til{1}_{\sigma^{(\alpha)}}(\tau):=\sum_{\beta=1}^rk_{\tau^{(\alpha)}\sigma^{(\alpha)}}^{-1}h_{\sigma^{(\alpha)}\sigma^{(\beta)}}(e)z^{m_{\tau}(\sigma^{(\alpha)})-m_{\tau}(\sigma^{(\beta)})}1_{\sigma^{(\beta)}}(\tau).$$
 By Condition (H1), $\{\til{1}_{\sigma^{(\alpha)}}(\tau)\}_{\alpha=1}^r$ gives a frame for $\cu{E}_{\sigma}(\tau)$.
 We prove that if $g:\{v\}\to\tau$ is a vertex, then
 $$\til{H}_{\ol{s}}(g):\til{1}_{\sigma^{(\alpha)}}(\tau)\mapsto F_{\ol{s}}(g)^*\til{1}_{\sigma^{(\alpha)}}(v).$$
 The coefficient of $\til{H}_{\ol{s}}(g)(\til{1}_{\sigma^{(\alpha)}}(\tau))$ attached to the base vector $F_{\ol{s}}(g)^*1_{\sigma^{(\gamma)}}(v)$ equals to
 $$\sum_{\beta=1}^rk_{\tau^{(\alpha)}\sigma^{(\alpha)}}^{-1}k_{v^{(\beta)}\tau^{(\beta)}}^{-1}h_{\sigma^{(\alpha)}\sigma^{(\beta)}}(e)h_{\sigma^{(\beta)}\sigma^{(\gamma)}}(g)\frac{\ol{s}_{\{v\}\to\tau}(m_v(\sigma^{(\beta)})-m_v(\sigma^{(\gamma)}))}{\ol{s}_{\{v\}\to\tau}(m_{\tau}(\sigma^{(\beta)}))}z^{m_{\tau}(\sigma^{(\alpha)})-m_{\tau}(\sigma^{(\beta)})+m_v(\sigma^{(\beta)})-m_v(\sigma^{(\gamma)})}$$
 By (H1), $\sigma^{(\alpha)},\sigma^{(\beta)}$ contain a common lift of $\tau$, we must have $\tau^{(\beta)}=\tau^{(\alpha)}$ and so $v^{(\beta)}=v^{(\alpha)}$. In particular,
 $$m_{\tau}(\sigma^{(\alpha)})-m_{\tau}(\sigma^{(\beta)})+m_v(\sigma^{(\beta)})-m_v(\sigma^{(\gamma)})=m_v(\sigma^{(\alpha)})-m_v(\sigma^{(\gamma)}).$$
 Using the formula $\ol{s}_g(m_{\tau}(\sigma^{(\alpha)}))=k_{v^{(\alpha)}\tau^{(\alpha)}}k_{\tau^{(\alpha)}\sigma^{(\alpha)}}k_{v^{(\alpha)}\sigma^{(\alpha)}}^{-1}$ and (H3), the sum becomes
 $$\sum_{\beta=1}^rk_{v^{(\alpha)}\sigma^{(\alpha)}}^{-1}h_{\sigma^{(\alpha)}\sigma^{(\gamma)}}(e\circ g)\ol{s}_{\{v\}\to\tau}(m_v(\sigma^{(\alpha)})-m_v(\sigma^{(\gamma)}))z^{m_v(\sigma^{(\alpha)})-m_v(\sigma^{(\gamma)})},$$
 which is the coefficient attached to $F_{\ol{s}}(g)^*1_{\sigma^{(\gamma)}}(v)$ in $F_{\ol{s}}(g)^*\til{1}_{\sigma^{(\alpha)}}(v)$. Hence $\{\til{1}_{\sigma^{(\alpha)}}(v)\}_{v\in\sigma}$ glue to a frame and gives a trivialization $\psi_{\sigma}:\cu{E}_0|_{V_{\ol{s}}(\sigma)}\to\cu{O}_{V_{\ol{s}}(\sigma)}^{\oplus r}$.
\end{proof}

\begin{remark}\label{rmk:trivialization}
The proof of Theorem \ref{thm:close} shows that for each $\sigma\in\msc{P}_{max}$, the trivialization
$$\psi_{\sigma}:i_{\sigma}^*\cu{E}_0\xrightarrow{\sim}\bigoplus_{\sigma'\in\msc{P}'(\sigma)}\cu{O}_{V_{\ol{s}}(\sigma)}^{\oplus\mu(\sigma')},$$
is explicitly given by mapping $\{\til{1}_{\sigma^{(\alpha)}}(v)\}_{v\in\tau}$ to $1_{\sigma^{(\alpha)}}$. Let $\sigma_1,\sigma_2\in\msc{P}_{max}$ and $\tau=\sigma_1\cap\sigma_2$. With respective to this frame, the transition map $\psi_{\sigma_2}\circ\psi_{\sigma_1}^{-1}$, in terms of coordinates of the open subset $V(\tau)\subset V(\sigma_1)$, is given by
$$\psi_{\sigma_2}\circ\psi_{\sigma_1}^{-1}:1_{\sigma_1^{(\alpha)}}\mapsto\sum_{\beta=1}^r\til{g}_{\sigma_1^{(\alpha)}\sigma_2^{(\beta)}}(s)z^{m_v(\sigma_1^{(\alpha)})-m_v(\sigma_2^{(\beta)})}1_{\sigma_2^{(\beta)}},$$
for some $\til{g}_{\sigma_1^{(\alpha)}\sigma_2^{(\beta)}}(s)\in\bb{C}$, depending on the gluing data $s$. We emphasis that in the sum, we have $$z^{m_v(\sigma_1^{(\alpha)})-m_v(\sigma_2^{(\beta)})}|_{V_{\{v\}\to\tau}}=0$$
if $m_v(\sigma_1^{(\alpha)})-m_v(\sigma_2^{(\beta)})\notin K_{\{v\}\to\tau}^{\vee}\cap\Lambda_{\sigma_1}^*$. 
\end{remark}

We combine the data ${\bf{h}}$ in Definition \ref{def:comp_data} and the data ${\bf{k}}_s$ obtained in Theorem \ref{thm:close}, and simply write ${\bf{h}}_s$ as this is the only data needed for the cocycle condition (\ref{eqn:cocycle_H}) to be satisfied. We denote the locally free sheaf obtained in Theorem \ref{thm:close} by $\cu{E}_0(\{\varphi_{\tau'}\},{\bf{D}}_s)$ for instance, where ${\bf{D}}_s$ is the data $({\bf{g}},{\bf{h}}_s)$. One would of course ask for the dependence of $\cu{E}_0(\{\varphi_{\tau'}\},{\bf{D}}_s)$ on the local representatives $\{\varphi_{\tau'}\}$ and ${\bf{D}}_s$. It is not hard to see that if $\{\varphi_{\tau'}'\}$ is an other choice of representative of $\varphi$, there exists another data ${\bf{D}}_s'=({\bf{g}},{\bf{h}},{\bf{k}}_s')$ such that $\cu{E}_0(\{\varphi_{\tau'}\},{\bf{D}}_s)=\cu{E}_0(\{\varphi_{\tau'}'\},{\bf{D}}_s')$. To prove this, first note that for each $\tau\in\msc{P}$ and each lift $\tau'$ of it, $G_{\sigma_1\sigma_2}(\tau')$ is independent of the choice of local representative of $\varphi$. It remains to consider the gluing maps $\{\til{H}_{\ol{s}}(g)\}$. For each $\tau'\in\msc{P}$, $f_{\tau'}:=\varphi_{\tau'}'-\varphi_{\tau'}$ is an affine function defined on $W_{\tau'}$. Recall that we have
$$\ol{s}_{\tau_1'\tau_2'\tau_3'}=k_{\tau_1'\tau_2'}k_{\tau_2'\tau_3'}k_{\tau_1'\tau_3'}^{-1}.$$
If we define
$$k_{\tau_1'\tau_2'}'=k_{\tau_1'\tau_2'}\ol{s}_g(f_{\tau_2'})\in\bb{C}^{\times},$$
then
$$\ol{s}_{g_1}(m_{\tau_2}'(\sigma^{(\alpha)}))\ol{s}_{g_2}(m_{\tau_3}'(\sigma^{(\alpha)}))\ol{s}_{g_3}(m_{\tau_3}'(\sigma^{(\alpha)}))^{-1}=k_{\tau_1'\tau_2'}'k_{\tau_2'\tau_3'}'k_{\tau_1'\tau_3'}'^{-1}.$$
Thus if we modify $H_{\ol{s}}(g)$ by ${\bf{k}}_s':=\{k_{\tau_1'\tau_2'}'\}$ as in the proof of Theorem \ref{thm:close}, we have
$$\til{H}_{\ol{s}}'(g)=\til{H}_{\ol{s}}(g).$$
Thus we have $\cu{E}_0(\{\varphi_{\tau'}\},{\bf{D}}_s)=\cu{E}_0(\{\varphi_{\tau'}'\},{\bf{D}}_s')$.
    
\begin{definition}\label{def:unobstructed}
    Let $\bb{L}$ be a tropical Lagrangian multi-section define over $(B,\msc{P})$ and $s$ be an open gluing data for the fan picture and $\ol{s}$ be its associated closed gluing data. Suppose $o_{\bb{L}}([s])=1$ and denote the data $({\bf{g}},{\bf{h}}_s)$ by ${\bf{D}}_s$, where ${\bf{g}}$ as in Definition \ref{def:comp_data} and ${\bf{h}}_s=({\bf{h}},{\bf{k}}_s)$ as in Theorem \ref{thm:close}. We denote the locally free sheaf obtained in Theorem \ref{thm:close} by $\cu{E}_0(\bb{L},{\bf{D}}_s)$. The set of all ${\bf{D}}_s$ is denoted by $\msc{D}_s(\bb{L})$. We say $\bb{L}$ is \emph{unobstructed} if $\msc{D}_s(\bb{L})\neq\emptyset$ and a pair $(\bb{L},{\bf{D}}_s)$ is a called a \emph{tropical Lagrangian brane}.
\end{definition}

\begin{remark}
 One can enrich $\bb{L}$ by a $\bb{C}^{\times}$-local system $\cu{L}$ on the domain $L$. Regarding it as a constructible sheaf on $L$, we obtain a set of specialization maps $\{f_{\tau_1'\tau_2'}\}_{\tau_1'\subset\tau_2'}\subset\bb{C}^{\times}$ that represent $\cu{L}$. Given ${\bf{D}}=({\bf{g}},{\bf{h}})\in\msc{D}(\bb{L})$, one can twist the gluing data ${\bf{h}}$ by setting
 $$h_{\sigma^{(\alpha)}\sigma^{(\beta)}}^{\cu{L}}(g):=h_{\sigma^{(\alpha)}\sigma^{(\beta)}}(g)f_{\tau_1'\tau_2'},$$
 where $g:\tau_1\to\tau_2$ and $\tau_1',\tau_2'\in\msc{P}'$ are lifts of $\tau_1,\tau_2$ that are uniquely determined by requiring  $\tau_1'\subset\tau_2'\subset\sigma^{(\alpha)}\cap\sigma^{(\beta)}$. By the definition that $h_{\sigma^{(\alpha)}\sigma^{(\beta)}}(g)\neq 0$ only if $\sigma^{(\alpha)},\sigma^{(\beta)}$ contain a common lift of $\tau_1$, it is easy to see that ${\bf{D}}^{\cu{L}}:=({\bf{g}},{\bf{h}}^{\cu{L}})\in\msc{D}(\bb{L})$.
\end{remark}

\section{Tropical locally free sheaves and their associated tropical Lagrangian multi-section}\label{sec:E_to_L}

As we have seen in the construction of $\cu{E}_0(\bb{L},{\bf{D}}_s)$, the restriction of $\cu{E}_0(\bb{L},{\bf{D}}_s)$ to a strata $X_{\tau}\subset X_0(B,\msc{P},s)$ is actually a toric vector bundle whose equivariant structure is determined by the fan structure $S_{\tau'}:W_{\tau'}\to L_{\tau'}$ and the choice of local representative of $\varphi|_{W_{\tau'}}$. See Remark \ref{rem:equiv_str} for the description of the equivariant structure. Recall that a strata $X_{\tau}$ of a toric variety $X_{\Sigma}$ is actually a closed orbit in $X_{\Sigma}$ with respective to the big torus action. Hence if $\cu{E}$ is a toric vector bundle on $X_{\Sigma}$, its restriction $\cu{E}|_{X_{\tau}}$ admits an induced big torus action and the inclusion map $\cu{E}|_{X_{\tau}}\to\cu{E}$ is equivariant with respective to the big torus action. It guides us to look at the following type of locally free sheaf over $X_0(B,\msc{P},s)$.

\begin{definition}\label{def:tropical}
    Let $\cu{E}_0$ be a locally free sheaf on $X_0(B,\msc{P},s)$ and for $\tau\in\msc{P}$, put $\cu{E}(\tau):=q_{\tau}^*\cu{E}_0$. A \emph{tropical structure} on $\cu{E}_0$ is a choice of toric vector bundle structure on $\cu{E}(\tau):=q_{\tau}^*\cu{E}_0$ such that for any $g:\tau_1\to\tau_2$ and toric indecomposable summand $\cu{E}^{(\alpha)}(\tau_2)$ of $\cu{E}(\tau_2)$, there exists character $\chi_g$ on $X_{\tau_1}$ such that the embedding $\cu{E}^{(\alpha)}(\tau_2)\hookrightarrow F_{\ol{s}}(g)^*(\cu{E}(\tau_1)\otimes(\chi_g))$ is $\cu{Q}_{\tau_1}\otimes\bb{C}^{\times}$-equivariant. A locally free sheaf on $X_0(B,\msc{P},s)$ that admits a tropical structure is called a \emph{tropical locally free sheaf}.
\end{definition}

Definition \ref{def:tropical} makes sense because of the following results.

\begin{theorem}[=Theorem 1.2.3 + Corollary 1.2.4 in \cite{Kly}]\label{thm:kyl_shift}
    Let $E,F$ be two toric vector bundles over a complete toric variety $X_{\Sigma}$. Then the following statement are true.
    \begin{enumerate}
        \item $E$ is indecomposable torically if and only if it is indecomposable as a ordinary vector bundle.
        \item If $E$ is a indecomposable summand of $F$, then there exists a character $\chi$ such that $E\otimes(\chi)$ is a toric summand of $F$.
    \end{enumerate}
    In particular, two indecomposable toric vector bundles $E,F$ on a toric variety are isomorphic as ordinary vector bundles if and only if $E\cong F\otimes(\chi)$ as toric vector bundles, for some character $\chi$.
\end{theorem}

\begin{proposition}[=Proposition 1.2.6 in \cite{Kly}]\label{prop:kly_summand_unique}
Indecomposable summands of a toric vector bundle over a complete toric variety is unique up to reordering.
\end{proposition}

As a whole, we obtain a classification of toric vector bundle structures on a vector bundle.

\begin{corollary}\label{cor:shift}
Let $E$ be a toric vector over a complete toric variety $X_{\Sigma}$. Then indecomposable summands of $E$ are toric vector bundle. Suppose $F$ is a toric vector bundle such that $F\cong E$ as ordinary vector bundles. Then by shifting indecomposable summands of $F$, we have $F\cong E$ as toric vector bundles.
\end{corollary}
\begin{proof}
    Indecomposable summands of a vector bundle over complete reduced scheme are unique. Hence they must be the toric indecomposable summands of $E$ by (1) in Theorem \ref{thm:kyl_shift}. If $F\cong E$ as ordinary vector bundles, then their indecomposable summands are isomorphic. By (2) in Theorem \ref{thm:kyl_shift}, they are torically isomorphic up to shift of characters. 
\end{proof}

The construction in Section \ref{sec:constructing_E_0} indeed gave us tropical locally sheaves.

\begin{proposition}\label{prop:tropical}
Let $(\bb{L},{\bf{D}}_s)$ be a tropical Lagrangian brane. Then $\cu{E}_0(\bb{L},{\bf{D}}_s)$ is tropical.
\end{proposition}
\begin{proof}
    Choose any representative $\{\varphi_{\tau'}\}$ of $\varphi$ to give $\cu{E}({\bf{g}}(\tau'))$ a structure of toric vector bundle over the strata $X_{\tau}$. For $\tau_1'\subset\tau_2'$, recall that the gluing isomorphism $H_{\ol{s}}(g):\cu{E}({\bf{g}}(\tau_2))\to F_{\ol{s}}(g)^*\cu{E}({\bf{g}}(\tau_1))$ takes the form
    $$1_{\sigma^{(\alpha)}}(\tau_2')\mapsto\sum_{\beta=1}^rk_{\tau_1^{(\alpha)}\tau_2^{(\alpha)}}^{-1}h_{\sigma^{(\alpha)}\sigma^{(\beta)}}(g)z^{m_{\tau_1'}(\sigma^{(\alpha)})-m_{\tau_1'}(\sigma^{(\beta)})}F_{\ol{s}}(g)^*1_{\sigma^{(\beta)}}(\tau_1')$$
    on the chart $V_{\tau_2\to\sigma}\subset X_{\tau_2}$ with respective to equivariant frames. Let $\lambda\in\cu{Q}_{\tau_1}\otimes\bb{C}^{\times}$. Applying $\lambda$ to the left hand side, we have
    $$\lambda\cdot 1_{\sigma}^{(\alpha)}(\tau_2')=\lambda^{p_g^*m_{\tau_2'}(\sigma^{(\alpha)})}1_{\sigma}^{(\alpha)}(\tau_2'),$$
    while when $\lambda$ is applied to the right hand side, we have
    $$\lambda^{m_{\tau_1'}(\sigma^{(\alpha)})}\sum_{\beta=1}^rk_{\tau_1^{(\alpha)}\tau_2^{(\alpha)}}^{-1}h_{\sigma^{(\alpha)}\sigma^{(\beta)}}(g)z^{m_{\tau_1'}(\sigma^{(\alpha)})-m_{\tau_1'}(\sigma^{(\beta)})}F_{\ol{s}}(g)^*1_{\tau_1'}^{(\beta)}(\sigma).$$
    As $\tau_1'\subset\tau_2'$, the slope difference $f:=p_g^*m_{\tau_2'}(\sigma')-m_{\tau_1'}(\sigma')$ is independent of $\sigma'$ as long as $\sigma'\supset\tau_2'$, which means $f\in\cu{Q}_{\tau_1}^*$. This affine function gives a character $\chi_f$ on $X_{\tau_1}$. Then it is easy to see that the map $H_{\ol{s}}(g)|_{\cu{E}({\bf{g}}(\tau_2'))}:\cu{E}({\bf{g}}(\tau_2'))\to F_{\ol{s}}(g)^*(\cu{E}({\bf{g}}(\tau_1'))\otimes(\chi_f))$ is a $\cu{Q}_{\tau_1}\otimes\bb{C}^{\times}$-equivariant embedding. In particular, $H_{\ol{s}}(g)$ is equivariant on any indecomposable summands of $\cu{E}({\bf{g}}(\tau_2'))$.
\end{proof}

Given a tropical locally free sheaf $\cu{E}_0$ on $X_0(B,\msc{P},s)$, we now construct a tropical Lagrangian multi-section over $(B,\msc{P})$. Let $\tau\in\msc{P}$. By assumption, $\cu{E}(\tau):=q_{\tau}^*\cu{E}_0$ admits a structure of toric vector bundle over the toric strata $X_{\tau}$. Let $\cu{E}^{(\alpha)}(\tau)$ be an indecomposable summand of $\cu{E}(\tau)$ and define
$$\tau^{(\alpha)}:=\tau\times\{\cu{E}^{(\alpha)}(\tau)\}.$$
As $\tau^{(\alpha)}\cong\tau$ via the first projection, we also refer them as \emph{cells}. From now on, we write $\cu{E}^{(\alpha)}(\tau)$ as $\cu{E}(\tau^{(\alpha)})$ and $\mu_{\tau^{(\alpha)}}$ for the multiplicity of $\cu{E}^{(\alpha)}(\tau)$ in $\cu{E}(\tau)$. Let $\msc{P}'(\tau)$ be the collection of all cells (counting with multiplicity) with the first projection being $\tau$. For $g:\tau_1\to\tau_2$, we define $\mf{p}_{\tau_2\tau_1}:\msc{P}'(\tau_2)\to\msc{P}'(\tau_1)$ by mapping $\tau_2'$ to $\tau_1'$ for which $\cu{E}(\tau_2')$ is summand of $F_{\ol{s}}(g)^*\cu{E}(\tau_1')$ via the equality $\cu{E}(\tau_2)=F_{\ol{s}}(g)^*\cu{E}(\tau_1)$. Since $\cu{E}_0$ is a global sheaf on $X_0(B,\msc{P},s)$, it is clear that
$$\mf{p}_{\tau_2\tau_1}\circ\mf{p}_{\tau_3\tau_2}=\mf{p}_{\tau_3\tau_1},$$
whenever $\tau_1\subset\tau_2\subset\tau_3$. Define $\mu:\msc{P}'\to\bb{Z}_{>0}$ by
$$\mu(\tau'):=\rk(\cu{E}(\tau')).$$
Then it is clear that $(\{\msc{P}'(\tau)\}_{\tau\in\msc{P}},\{\mf{p}_{\tau_2\tau_1}\}_{\tau_1\subset\tau_2},\mu)$ defines an abstract branched covering (see Appendix \ref{sec:appA} for this notion) over $(B,\msc{P})$. By the construction in Appendix \ref{sec:appA}, the data $(\{\msc{P}'(\tau)\}_{\tau\in\msc{P}},\{\mf{p}_{\tau_2\tau_1}\}_{\tau_1\subset\tau_2},\mu)$ induce a branched covering map of tropical spaces and we denote it by $\pi_{\cu{E}_0}:(L_{\cu{E}_0},\msc{P}_{\cu{E}_0}',\mu_{\cu{E}_0})\to(B,\msc{P})$.

It remains to construct the fan structure and the piecewise linear function. Let's first make the following

\begin{definition}
    Let $\Sigma$ be a complete fan. Two tropical Lagrangian multi-sections $\bb{L}_1,\bb{L}_2$ over $\Sigma$ is said to be \emph{differ by a shift of affine function} if there exists an isomorphism of weighted cone complexes $f:(L_1,\Sigma_1,\mu_1)\to(L_2,\Sigma_2,\mu_2)$ such that $\pi_1=\pi_2\circ f$ and $f^*\varphi_2-\varphi_1$ is an affine function on $L_1$.
\end{definition}

Let $\bb{L}_{\tau'}$ be the associated tropical Lagrangian multi-section (see \cite{branched_cover_fan} or \cite{Suen_trop_lag} for the construction) of an indecomposable summand $\cu{E}(\tau')$, which is always separable (Definition 3.13 and Proposition 3.21 in \cite{Suen_trop_lag}). By Theorem \ref{thm:kyl_shift}, different choice of equivariant structure on $\cu{E}(\tau')$ only leads to a shift of affine function on $L_{\tau'}$.

For $e:\tau\to\sigma$ and $\tau'\subset\sigma'$, $\cu{E}(\sigma')$ is by definition an indecomposable summand of $F_{\ol{s}}(e)^*\cu{E}(\tau')$. Hence by Theorem \ref{thm:kyl_shift}, up to a shift of affine function if necessary, the tropical Lagrangian multi-section $\bb{L}_{\sigma'}$ is a localization (see Appendix \ref{sec:appB} for this notion) of $\bb{L}_{\tau'}$ along some cone in $\Sigma_{\tau'}$ whose projection to $\Sigma_{\tau}$ is $K_{\tau}(\sigma)$. By Theorem \ref{thm:bundle_localization}, such cone is unique as $\bb{L}_{\tau'}$ is separable. Define $S_{\tau'}|_{\sigma'\cap W_{\tau'}}:\sigma'\cap W_{\tau'}\to K_{\tau'}(\sigma')$ by
$$x'\mapsto ((S_{\tau}\circ\pi)(x'),m_{\tau'}(\sigma')),$$
for $x'\in\sigma'\cap W_{\tau'}$ and $m_{\tau'}(\sigma')$ the slope of $\varphi_{\tau'}$ on the cone $K_{\tau'}(\sigma')$. Since $\varphi_{\tau'}$ is continuous, it is not hard to see that $S_{\tau'}$ is well-defined and continuous. This gives the desired fan structure $S_{\tau'}:W_{\tau'}\to|L_{\tau'}|$. Define $\varphi_{\cu{E}_0}:=\{S_{\tau'}^*\varphi_{\tau'}\}_{\tau'\in\msc{P}'}$. By Theorem \ref{thm:kyl_shift}, shifting $\bb{L}_{\tau_2'}$ by an affine function if necessary, we may assume it is a localization of $\bb{L}_{\tau_1'}$ if $\cu{E}(\tau_2')$ is a summand of $\cu{E}(\tau_1')|_{X_{\tau_2}}$, which means $\varphi_{\cu{E}_0}\in H^0(L_{\cu{E}_0},\cu{MPL}_{\msc{P}_{\cu{E}_0}'})$.

\begin{definition}
    Let $\cu{E}_0$ be a tropical locally free sheaf over $X_0(B,\msc{P},s)$. The data $\bb{L}_{\cu{E}_0}:=(L_{\cu{E}_0},\msc{P}_{\cu{E}_0}',\mu_{\cu{E}_0},\pi_{\cu{E}_0},\varphi_{\cu{E}_0})$ is called the \emph{associated tropical Lagrangian multi-section of $\cu{E}_0$}.
\end{definition}

Given a tropical locally free sheaf $\cu{E}_0$, we can also associate a data ${\bf{D}}_s(\cu{E}_0)=({\bf{g}},{\bf{h}}_s)\in\msc{D}_s(\bb{L}_{\cu{E}_0})$ such that
$$\cu{E}_0\cong\cu{E}_0(\bb{L}_{\cu{E}_0},{\bf{D}}_s(\cu{E}_0)).$$
Indeed, ${\bf{g}}$ exists because of the fact that each $\cu{E}(\tau')$ is toric. For ${\bf{h}}_s$, note that for $g:\tau_1\to\tau_2$, by definition, $\cu{E}(\tau_2')$ is mapped to a summand of $F_{\ol{s}}(g)^*\cu{E}(\tau_1')$ for a unique $\tau_1'$. Let $\{1_{\sigma}^{(\alpha)}(\tau_1')\},\{1_{\sigma}^{(\alpha)}(\tau_2')\}$ be equivariant frame of $\cu{E}(\tau_1'),\cu{E}(\tau_2')$, respectively. By definition, there is a character $\chi_g$ so that the natural map $H_{\ol{s}}(g):\cu{E}(\tau_2')\to F_{\ol{s}}(g)^*(\cu{E}(\tau_1')\otimes(\chi_g))$ is $\cu{Q}_{\tau_1}\otimes\bb{C}^{\times}$-equivariant. Let $1_{\chi_g}(\sigma)$ be an equivariant frame of $(\chi_g)$ on the chart $V_{\tau_1\to\sigma}\subset X_{\tau_1}$. Then $H_{\ol{s}}(g)$ is of the form
$$1_{\sigma}^{(\alpha)}(\tau_2')\mapsto\sum_{\beta=1}^{\mu(\tau_1')}h_{\sigma^{(\alpha)}\sigma^{(\beta)}}(g)_sz^{m_{\tau_1'}(\sigma^{(\alpha)})-m_{\tau_1'}(\sigma^{(\beta)})}F_{\ol{s}}(g)^*(1_{\sigma}^{(\beta)}(\tau_1')\otimes 1_{\sigma}(\chi_g)),$$
for some $h_{\sigma^{(\alpha)}\sigma^{(\beta)}}(g)_s\in\bb{C}$ that is non-zero only if $m_{\tau_1'}(\sigma^{(\alpha)})-m_{\tau_1'}(\sigma^{(\beta)})\in K_{\tau_2\to\sigma}^{\vee}\cap\cu{Q}_{\tau_2}^*$. As $(\chi_g)$ is just the trivial line bundle on $X_{\tau_1}$, we have a well-define map
$$1_{\sigma}^{(\alpha)}(\tau_2')\mapsto\sum_{\beta=1}^{\mu(\tau_1')}h_{\sigma^{(\alpha)}\sigma^{(\beta)}}(g)_sz^{m_{\tau_1'}(\sigma^{(\alpha)})-m_{\tau_1'}(\sigma^{(\beta)})}F_{\ol{s}}(g)^*1_{\sigma}^{(\beta)}(\tau_1').$$
This gives the data ${\bf{h}}_s$. On the other hand, the assignment
$$\bb{L}\mapsto\cu{E}_0(\bb{L},{\bf{D}}_s)\mapsto\bb{L}_{\cu{E}_0(\bb{L},{\bf{D}}_s)}$$
rarely be the identity, even in the case of one toric piece \cite{Suen_trop_lag}. This indicates the fact that non-Hamiltonian equivalent Lagrangian branes can still be equivalent in the derived Fukaya category (see \cite{CS_SYZ_imm_Lag} for this phenomenon). In our case, although $\bb{L}$ and $\bb{L}_{\cu{E}_0(\bb{L},{\bf{D}}_s)}$ are not isomorphic, they are related by a covering morphism (c.f. Definition \ref{def:covering_morphism}).

\begin{proposition}\label{prop:L_L_E}
    Let $\cu{E}_0:=\cu{E}_0(\bb{L},{\bf{D}}_s)$. There is a covering morphism $f:\bb{L}_{\cu{E}_0}\to\bb{L}$.
\end{proposition}
\begin{proof}
    For $\tau'\in\msc{P}'$, let $S(\tau')$ be the set of all lifts $\tau^{(\alpha)}\in\msc{P}_{\cu{E}_0}'$ of $\tau\in\msc{P}$ so that
    $$\bigoplus_{\tau^{(\alpha)}\in S(\tau')}\cu{E}(\tau^{(\alpha)})=\cu{E}({\bf{g}}(\tau')).$$
    Then $f$ is defined to by mapping all cells in $S(\tau')$ to $\tau'\in\msc{P}'$. It is clear from the definition of $f$ that it is continuous, preserves polyhedral compositions, the piecewise linear functions and satisfies $Tr_f(\mu_{\cu{E}_0})=\mu$. 
\end{proof}

We end this section by proving the following

\begin{theorem}\label{thm:many_examples}
    Let $\Xi\subset N_{\bb{R}}$ be a polytope centered at the origin so that the natural affine structure with singularities on the boundary $B:=\partial\Xi$ is integral and $\Sigma$ be the fan obtained by taking cones of proper faces of $\Xi$. Let $X_{\Sigma}$ be the projective toric variety associated to a $\Sigma$ and $X_0:=\partial X_{\Sigma}$ the toric boundary. Then $\cu{E}_0:=\cu{E}|_{X_0}$ is a tropical locally free sheaf.
\end{theorem}
\begin{proof}
Let $\msc{P}$ be given by proper faces of $\Xi$. The affine structure on $B$ is given by the fan structure
$$W_v\subset N_{\bb{R}}\to N_{\bb{R}}/\bb{R}\inner{v},$$
for $v\in\Xi$ a vertex. By assumption, this affine structure with singularities is integral and we have $X_0=X_0(B,\msc{P})$. Since each strata of $X_0$ is a toric strata of $X_{\Sigma}$, each $\cu{E}(\tau)$ has a natural $N\otimes\bb{C}^{\times}$-equivariant structure. Since $\cu{E}$ is toric, it has the corresponding tropical Lagrangian multi-section $\bb{L}_{\Sigma}$ over $\Sigma$. Denote the piecewise linear function on $\bb{L}_{\Sigma}$ by $\varphi$. For each proper face $\tau\in\msc{P}$, let $K(\tau)\in\Sigma$ be the corresponding cone. Choose a splitting $\iota_{\tau}$ of the projection $p_{\tau}:N\to N/(\bb{R} K(\tau)\cap N)\cong\cu{Q}_{\tau}$ to equip $\cu{E}(\tau)$ a $\cu{Q}_{\tau}\otimes\bb{C}^{\times}$-equivariant structure over $X_{\tau}$. Let $\cu{E}^{(\alpha)}(\tau)$ be a toric indecomposable summand of $\cu{E}(\tau)$ and $\bb{L}_{\tau}^{(\alpha)}$ be its associated tropical Lagrangian multi-section. It has been shown in Theorem \ref{thm:bundle_localization} that $\bb{L}_{\tau_2}^{(\alpha)}$ is a localization of $\bb{L}_{\Sigma}$ and
$$f_{\tau}:=p_{\tau}^*\varphi_{\tau}^{(\alpha)}-\varphi$$
is an affine function in a neighborhood of the cone that we localized. Therefore, for $g:\tau_1\to\tau_2$, $$p_{\tau_1}^*\varphi_{\tau_1}^{(\beta)}-p_{\tau_2}^*\varphi_{\tau_2}^{(\alpha)}$$
is also an affine function as long as $\cu{E}^{(\alpha)}(\tau_2)\subset F_{\ol{s}}(g)^*\cu{E}^{(\beta)}(\tau_1)$. Note that we have $p_{\tau_2}=p_g\circ p_{\tau_1}$ and $\iota_{\tau}^*$ is a left inverse of $p_{\tau}^*$. By applying $\iota_{\tau_1}^*$, the difference
$$\varphi_{\tau_1}^{(\beta)}-p_g^*\varphi_{\tau_2}^{(\alpha)}$$
is an affine function. This affine function gives a character $\chi_g$ on $X_{\tau_1}$ so that $\cu{E}^{(\alpha)}(\tau_2)\subset F_{\ol{s}}(g)^*\cu{E}^{(\beta)}(\tau_1)$ is a $\cu{Q}_{\tau_1}\otimes\bb{C}^{\times}$-equivariant 
embedding. This completes the proof of the theorem.
\end{proof}

Theorem \ref{thm:many_examples} provides us an abundant source of examples of tropical locally free sheaves that is \emph{smoothable}. Indeed, when $X_0$ is the central fiber of a family of Calabi-Yau hypersurfaces $\cu{X}\subset X_{\Sigma}$. By simply restricting the toric vector bundle $\cu{E}$ on $X_{\Sigma}$ to the family $\cu{X}$, the pair $(X_0,\cu{E}_0)$ is tautologically smoothable to $(\cu{X}_t,\cu{E}|_{\cu{X}_t})$. Although the smoothing problem in this case is trivial, it does provide us many examples of tropical locally free sheaves on $X_0$ that can be smoothed in \emph{any} dimension.

\section{The correspondence}\label{sec:equiv}

In this section, we would like to establish a correspondence between the set of tropical free sheaves modulo isomorphism and the set of tropical Lagrangian multi-sections modulo certain non-trivial equivalence.

If two tropical locally free sheaves $\cu{E}_0,\cu{E}_0'$ are isomorphic, then for any $\tau\in\msc{P}$, $\cu{E}(\tau)\cong\cu{E}'(\tau)$ as ordinary vector bundles on the strata $X_{\tau}$. By Corollary \ref{cor:shift}, their indecomposable summands are isomorphic as toric vector bundles up to shift of characters on $X_{\tau}$. Hence the associated tropical Lagrangian multi-sections of their indecomposable summand only differ from each other by shifts of affine functions. This gives a covering isomorphism $f:\bb{L}_{\cu{E}_0}\xrightarrow{\sim}\bb{L}_{\cu{E}_0'}$ between their associated tropical Lagrangian multi-sections. We define the equivalence between data in $\msc{D}_s(\bb{L})$ on a fixed tropical Lagrangian multi-section $\bb{L}$ tautologically.

\begin{definition}
    Let $\bb{L}$ be an unobstructed tropical Lagrangian multi-section over $(B,\msc{P})$ and ${\bf{D}}_s,{\bf{D}}_s'\in\msc{D}_s(\bb{L})$. We write ${\bf{D}}_s\sim{\bf{D}}_s'$ if $\cu{E}_0(\bb{L},{\bf{D}}_s)\cong\cu{E}_0(\bb{L},{\bf{D}}_s')$.
\end{definition}

Given two unobstructed tropical Lagrangian multi-sections $\bb{L}_1,\bb{L}_2$ over $(B,\msc{P})$ and a data ${\bf{D}}_s\in\msc{D}_s(\bb{L}_2)$. Suppose $f:\bb{L}_2\to\bb{L}_1$ is a covering isomorphism. It is easy to provide a data $f_*{\bf{D}}_s\in\msc{D}_s(\bb{L}_1)$ for $\bb{L}_1$ so that
\begin{equation}\label{eqn:E_1=E_2}
    \cu{E}_0(\bb{L}_1,f_*{\bf{D}}_s)\cong\cu{E}_0(\bb{L}_2,{\bf{D}}_s).
\end{equation}
In particular, if $\cu{E}_0\cong\cu{E}_0'$, then $f_*{\bf{D}}_s(\cu{E}_0)\sim{\bf{D}}_s(\cu{E}_0')$. However, as we have pointed out in Section \ref{sec:E_to_L} that non-isomorphic tropical Lagrangian multi-sections can still give rise to isomorphic tropical locally free sheaves after choosing suitable brane data. These tropical Lagrangian multi-sections should be regarded as equivalent objects in some sense. In the remaining part of this section, we explore this non-trivial equivalence.

\begin{definition}\label{def:order_L}
	Let $\bb{L}_1,\bb{L}_2$ be two tropical Lagrangian multi-sections of same degree over $(B,\msc{P})$. We write $\bb{L}_1\leq\bb{L}_2$ if there exists a covering morphism $f:\bb{L}_2\to\bb{L}_1$.
\end{definition}

Given a covering morphism $f:\bb{L}_2\to\bb{L}_1$ and data ${\bf{D}}_s\in\msc{D}_s(\bb{L}_2)$, we can define the \emph{push-forward $f_*{\bf{D}}_s$ of ${\bf{D}}_s$} by
$$(f_*g)_{f(\sigma_1^{(\alpha)})f(\sigma_2^{(\beta)})}(f(\tau')):=g_{\sigma_1^{(\alpha)}\sigma_2^{(\beta)}}(\tau'),\,(f_*h)_{f(\sigma^{(\alpha)})f(\sigma^{(\beta)})}(g):=h_{\sigma^{(\alpha)}\sigma^{(\beta)}}(g),$$
where $\sigma_1^{(\alpha)},\sigma_2^{(\beta)},\tau'$ are some choices of preimage cells of $f(\sigma_1^{(\alpha)}),f(\sigma_2^{(\beta)}),f(\tau')$ such that 
$$m_{f(\tau')}(f(\sigma_1^{(\alpha)}))=m_{\tau'}(\sigma_1^{(\alpha)}),\,m_{f(\tau')}(f(\sigma_2^{(\beta)}))=m_{\tau'}(\sigma_2^{(\beta)})$$
and $\tau'\subset\sigma_1^{(\alpha)}\cap\sigma_2^{(\beta)}$. It is straightforward to check that $f_*{\bf{D}}_s\in\msc{D}_s(\bb{L}_1)$. Different choices of preimage cells of $f(\sigma^{(\alpha)})$ amount a permutation of the ordered frame $\{1_{\sigma^{(\alpha)}}(\tau)\}_{\alpha=1}^r$ that preserve the $\cu{Q}_{\tau}\otimes\bb{C}^{\times}$-action. Hence such choice won't affect the resulting tropical locally free sheaf.

\begin{definition}\label{def:order}
	Let $(\bb{L}_1,{\bf{D}}_s^{(1)}),(\bb{L}_2,{\bf{D}}_2^{(2)})$ be two tropical Lagrangian branes. We write $(\bb{L}_1,{\bf{D}}_s^{(1)})\leq(\bb{L}_2,{\bf{D}}_s^{(2)})$ if there exists a covering morphism $f:\bb{L}_2\to\bb{L}_1$ such that $f_*{\bf{D}}_s^{(2)}\sim{\bf{D}}_s^{(1)}$.
\end{definition}

\begin{definition}
	Let $(\bb{L}_1,{\bf{D}}_s^{(1)}),(\bb{L}_2,{\bf{D}}_2^{(2)})$ be two tropical Lagrangian branes. We write $(\bb{L}_1,{\bf{D}}_s^{(1)})\sim(\bb{L}_2,{\bf{D}}_2^{(2)})$ if there exists a tropical Lagrangian brane $(\bb{L},{\bf{D}}_s)$ over $(B,\msc{P})$ such that $(\bb{L}_i,{\bf{D}}_s^{(i)})\leq(\bb{L},{\bf{D}}_s)$, for all $i=1,2$. We say $(\bb{L}_1,{\bf{D}}_s^{(1)})$ is \emph{combinatorially equivalent} to $(\bb{L}_2,{\bf{D}}_s^{(2)})$ if there exists a sequence of tropical Lagrangian branes $(\bb{L}_1',{\bf{D}}_s'^{(1)}),(\bb{L}_2',{\bf{D}}_s'^{(2)}),\dots,(\bb{L}_k',{\bf{D}}_s'^{(k)})$ over $(B,\msc{P})$ such that $(\bb{L}_1',{\bf{D}}_s'^{(1)})=(\bb{L}_1,{\bf{D}}_s^{(1)}),(\bb{L}_k',{\bf{D}}_s'^{(k)})=(\bb{L}_2,{\bf{D}}_s^{(2)})$ and $(\bb{L}_{i+1}',{\bf{D}}_s'^{(i+1)})\sim_c(\bb{L}_i',{\bf{D}}_s'^{(i)})$, for all $i=1,\dots,k-1$.
\end{definition}

\begin{remark}
	The relation $\sim_c$ is only reflexive and symmetric. The notion of combinatorially equivalence is the transitive closure of $\sim_c$ and hence, an equivalent relation. Geometrically, $\bb{L}_1,\bb{L}_2$ are combinatorially equivalent means one can fold or unfold cells of $\bb{L}_1$ to obtain $\bb{L}_2$ in finite steps.
\end{remark}

We define
\begin{align*}
    \text{TLFS}(X_0(B,\msc{P},s)):=&\,\frac{\{\text{Tropical locally free sheaves on }X_0(B,\msc{P},s)\}}{\text{isomorphism}}\\
    \text{TLB}(B,\msc{P},s):=&\,\frac{\{\text{Tropical Lagrangian branes over }(B,\msc{P},s)\}}{\text{combinatorial equivalence}}.
\end{align*}
By Proposition \ref{prop:L_L_E}, we have the following

\begin{theorem}\label{thm:bijection}
    We have a canonical bijection
    $$\cu{F}:\mathrm{TLFS}(X_0(B,\msc{P},s))\to\mathrm{TLB}(B,\msc{P},s),$$
    given by $\cu{E}_0\mapsto(\bb{L}_{\cu{E}_0},{\bf{D}}_s(\cu{E}_0))$. Its inverse is given by $(\bb{L}_{\cu{E}_0},{\bf{D}}_s)\mapsto\cu{E}_0(\bb{L}_{\cu{E}_0},{\bf{D}}_s)$.
\end{theorem}
\begin{proof}
    We have seen that the composition
    $$\cu{E}_0\mapsto(\bb{L}_{\cu{E}_0},{\bf{D}}_s(\cu{E}_0))\mapsto\cu{E}_0(\bb{L}_{\cu{E}_0},{\bf{D}}_s(\cu{E}_0))$$
    is the identity. It remains to shown that the assignment $\cu{E}_0\mapsto(\bb{L}_{\cu{E}_0},{\bf{D}}_s)$ is surjective, that is, given $(\bb{L},{\bf{D}}_s)$, whether $\bb{L}$ is combinatorially equivalent to $\bb{L}_{\cu{E}_0(\bb{L},{\bf{D}}_s)}$. This follows immediately from Proposition \ref{prop:L_L_E}.
\end{proof}

\begin{remark}
 Since every vector bundle on $\bb{P}^1$ splits in to direct sum of line bundles, we see that the ramification locus $S'$ of $\bb{L}_{\cu{E}_0}$ is always of codimension at least 2. In particular, every tropical Lagrangian brane is combinatorial equivalent to a tropical Lagrangian brane with $\codim(S')\geq 2$.
\end{remark}

\appendix

\section{Construction of branched covering maps of tropical spaces via discrete data}\label{sec:appA}

Let $B$ be an integral affine manifold with singularities and $\msc{P}$ a polyhedral decomposition. Branched covering of tropical spaces over $(B,\msc{P})$ can be constructed via discrete data.

\begin{definition}
    Let $B$ be an integral affine manifold with singularities and $\msc{P}$ be a polyhedral decomposition. A \emph{covering data} over $(B,\msc{P})$ is a triple $(\mf{P},\mf{p},\mu)$ that satisfies the following
    \begin{enumerate}
        \item $\mf{P}:=\{\msc{P}'(\tau)\}_{\tau\in\msc{P}}$ is a collection of finite sets, paramatrized by $\msc{P}$. We put $$\msc{P}':=\bigcup_{\tau\in \msc{P}'(\tau)}\msc{P}'(\tau).$$
        \item $\mf{p}:=\{\mf{p}_{\tau_1\tau_2}\}_{\tau_2\subset\tau_1}$ is a collection of surjections $\mf{p}_{\tau_1\tau_2}:\msc{P}'(\tau_1)\to \msc{P}'(\tau_2)$ such that for $\tau_3\subset\tau_2\subset\tau_1$, we have  $\mf{p}_{\tau_2\tau_3}\circ \mf{p}_{\tau_1\tau_2}=\mf{p}_{\tau_1\tau_3}$.
        \item $\mu:\msc{P}'\to\bb{Z}_{>0}$ is a function such that
        $$\sum_{\tau'\in \msc{P}'(\tau)}\mu(\tau')$$
        is a constant independent of $\tau\in\msc{P}$.
    \end{enumerate}
\end{definition}

We construct a branched covering map between tropical spaces as follows. Define a partial ordering $\subset'$ on $\msc{P}'$ by setting
$$\tau_1'\subset'\tau_2'\Longleftrightarrow\tau_1\subset\tau_2\text{ and }\mf{p}_{\tau_1\tau_2}(\tau_2')=\tau_1'.$$
Equip $\msc{P},\msc{P}'$ the poset topology, that is, a subset $\msc{Q}'\subset\msc{P}'$ is closed if and only if it satisfies
$$\sigma'\in\msc{Q}'\text{ and }\tau'\subset'\sigma'\Longrightarrow\tau'\in\msc{Q}'.$$
Then the map $\msc{P}'\to\msc{P}$ given by $\tau'\mapsto\tau$ is continuous. There is another map $B\to\msc{P}$ mapping $x\in B$ to $\tau_x\in\msc{P}$, the unique cell such that $x\in \Int(\tau_x)$, which is also continuous. Define the topological space
$$L:=B\times_{\msc{P}}\msc{P}'.$$
It is not hard to see that $L$ is in fact Hausdorff and paracompact. There is a collection of closed subsets $\msc{P}\times_{\msc{P}}\msc{P}'\cong\msc{P}'$. We can then write $\tau_1'\subset\tau_2'$ instead of $\tau_1'\subset'\tau_2'$ if we regard $\tau_1',\tau_2'$ as subsets in $L$. Let $\pi:L\to B$ be the first projection. It maps elements in $\msc{P}'$ homeomorphic to elements in $\msc{P}$. In particular, we can talk about the relative interior an element $\tau'\in\msc{P}'$, namely,
$$\Int(\tau'):=\pi^{-1}(\Int(\tau))\cap\tau'$$
Define $\mu:L\to\bb{Z}_{>0}$ by
$$\mu:x'\mapsto\mu(\tau_{x'}'),$$
where $\tau_{x'}'$ is the unique element in $\msc{P}'$ for which $x'\in \Int(\tau_{x'}')$. Define the \emph{sheaf of piecewise linear functions on $L$} to be the sheaf
$$\cu{PL}_{\msc{P}'}(U'):=\{\varphi\in C^0(U',\bb{R}):\varphi|_{U'\cap \Int(\sigma')}\circ\pi|_{U'\cap \Int(\sigma')}^{-1}\text{ is an affine function for all }\sigma'\in\msc{P}_{max}'\}$$
and the \emph{sheaf of affine functions on $L$} to be the sheaf associated to the presheaf
$$\cu{A}ff_L(U'):=\lim_{\substack{\longrightarrow\\U\supset\pi(U')}}Aff_B(U).$$
We have $\ul{\bb{R}}_L\subset\cu{A}ff_L\subset\cu{PL}_{\msc{P}'}$. Define $\pi^{\#}:\cu{PL}_{\msc{P}}\to\cu{PL}_{\msc{P}'}$ by pulling back a germ of piecewise linear functions on $B$ to $L$. It is clear that $\pi^{\#}$ preserves affine functions. Hence $\pi:(L,\msc{P}',\mu)\to(B,\msc{P})$ is a branched covering map between tropical spaces that preserves polyhedral decompositions.

\section{Localization}\label{sec:appB}

Let $\Sigma$ be a complete fan and $\bb{L}$ a tropical Lagrangian multi-section over it. We look at $\bb{L}$ in a neighborhood of a cone $\tau'\in\Sigma'$ and construct another tropical Lagrangian multi-section $\bb{L}_{\tau'}$ called the \emph{localization of $\bb{L}$ along $\tau'$}.

Fix $\tau\in\Sigma$. Let $\iota_{\tau}:N_{\tau}\to N$ be a lift of $p_{\tau}:N\to N_{\tau}$, which induces a projection $\iota_{\tau}^*:M\to M_{\tau}$ and a fan $\Sigma_{\tau}$ on $N_{\tau,\bb{R}}$. For any cone $\sigma'\in\Sigma'$ such that $\pi(\sigma')=\sigma\supset\tau$, the slope $m(\sigma')\in M(\sigma)$ gives an element $$m_{\tau}(\sigma'):=\iota_{\tau}^*m(\sigma')\in M_{\tau}/(\sigma^{\perp}\cap M_{\tau}),$$
Define $K_{\tau}(\sigma'):=K_{\tau}(\sigma)\times\{m_{\tau}(\sigma')\}\in\Sigma_{\tau}'$. Let $\tau'\in\Sigma'$ be a lift of $\tau$. Define
$$\text{Star}^{\circ}(\tau'):=\bigcup_{\sigma'\supset\tau'}\Int(\sigma'),$$
the open star of $\tau'$ and
$$\Sigma_{\tau'}':=\{K_{\tau}(\sigma')\,|\,\sigma'\supset\tau'\}.$$
We write an element in $\Sigma_{\tau'}'$ as $K_{\tau'}(\sigma')$ to emphasis its dependence on $\tau'$. This gives a topological space
$$L_{\tau'}:=|\Sigma_{\tau}|\times_{\Sigma_{\tau}}\Sigma_{\tau'}',$$
a projection $\pi_{\tau'}:L_{\tau'}\to|\Sigma_{\tau}|$, a multiplicity map $\mu_{\tau'}(K_{\tau'}(\sigma')):=\mu(\sigma')$, and a piecewise linear function $$\varphi_{\tau'}|_{K_{\tau'}(\sigma')}:=m_{\tau}(\sigma').$$ The slope of $\varphi_{\tau'}$ on a cone $K_{\tau'}(\sigma')\in\Sigma_{\tau'}'$ will be denoted by $m_{\tau'}(\sigma')\in M_{\tau}$, which is of course, equals to $m_{\tau}(\sigma')$. It is clear that $(L_{\tau'},\Sigma_{\tau'}',\mu_{\tau'},\pi_{\tau'},\varphi_{\tau'})$ is a tropical Lagrangian multi-section over $\Sigma_{\tau}$. We denote it by $\bb{L}_{\tau'}$. It's degree is given by $\mu(\tau')$.
    
To understand the relation between $\varphi$ and $\varphi_{\tau'}$, we define a projection $p_{\tau'}:\text{Star}^{\circ}(\tau')\to L_{\tau'}$ by setting
$$p_{\tau'}|_{\Int(\sigma')}:\Int(\sigma')\to K_{\tau}(\sigma)\times\{m_{\tau}(\sigma')\},$$
where the first component is given by $p_{\tau}\circ\pi|_{\text{Star}^{\circ}(\tau')}$. Clearly, $p_{\tau'}$ maps cones to cones and by construction, we have
$$\pi_{\tau'}\circ p_{\tau'}=p_{\tau}\circ \pi|_{\text{Star}^{\circ}(\tau')}.$$
For $\sigma_1',\sigma_2'\in\Sigma'(n)$ such that $\tau'\subset\sigma_1',\sigma_2'$, by continuity of $\varphi$, we have
$$m(\sigma_1')-m(\sigma_2')=p_{\tau}^*(m_{\tau'}),$$
for some $m_{\tau'}\in\tau^{\perp}\cap M=M_{\tau}$. As $\iota_{\tau}^*:M\to M_{\tau}$ is the left inverse of the inclusion $p_{\tau}^*:M_{\tau}\to M$, we have
$$m_{\tau'}(\sigma_1')-m_{\tau'}(\sigma_2')=\iota_{\tau}^*(m(\sigma_1')-m(\sigma_2'))=\iota_{\tau}^*p_{\tau}^*(m_{\tau'})=m_{\tau'}.$$
Hence
\begin{equation}\label{eqn:continous}
    p_{\tau}^*\left(m_{\tau'}(\sigma_1')-m_{\tau'}(\sigma_2')\right)=m(\sigma_1')-m(\sigma_2').
\end{equation}
For $\tau'$ and a maximal cone $\sigma'\supset\tau'$, the function $$f_{\tau'}:=p_{\tau}^*m_{\tau'}(\sigma')-m(\sigma')\in M$$
is independent of $\sigma'\supset\tau'$ and hence a linear function defined on $\text{Star}^{\circ}(\tau')$. Thus
\begin{equation}\label{eqn:diffence_affine}
    p_{\tau'}^*\varphi_{\tau'}=\varphi+f_{\tau'}
\end{equation}
on $\text{Star}^{\circ}(\tau')$. We can generalize this to arbitrary pair of stratum as follows. For $\tau_2'\subset\tau_1'$, we have $\text{Star}^{\circ}(\tau_1')\subset \text{Star}^{\circ}(\tau_2')$, so the difference $$p_{\tau_1'}^*\varphi_{\tau_1'}-p_{\tau_2'}^*\varphi_{\tau_2'}$$
is a linear function on $\text{Star}^{\circ}(\tau_1')$. As a whole, we proved the following
	
\begin{theorem}\label{thm:difference=affine}
For any $\tau\in\Sigma$, by choosing a lift of the projection $p_{\tau}:N\to N_{\tau}$, there is a collection of tropical Lagrangian multi-sections $\{\bb{L}_{\tau'}\}_{\tau':\pi(\tau')=\tau}$ over $\Sigma_{\tau}$ such that, for each lift $\tau'$ of $\tau$, there is a map $p_{\tau'}:\text{Star}^{\circ}(\tau')\to L_{\tau'}$ so that
$$\pi_{\tau'}\circ p_{\tau'}=p_{\tau}\circ \pi|_{\text{Star}^{\circ}(\tau')}.$$
Moreover, for $\tau_2'\subset\tau_1'$, the difference
$$p_{\tau_1'}^*\varphi_{\tau_1'}-p_{\tau_2'}^*\varphi_{\tau_2'}$$
is an integral linear function on $\text{Star}^{\circ}(\tau_1')$.
\end{theorem}
	
Recall the definition of separability introduced in \cite{Suen_trop_lag}, Definition 3.13.

\begin{definition}\label{def:separable}
	A tropical Lagrangian multi-section $\bb{L}=(L,\Sigma_L,\mu,\pi,\varphi)$ over a fan $\Sigma$ is said to be \emph{separable} if it satisfies the following condition: For any $\tau\in\Sigma$ and distinct lifts $\tau^{(\alpha)},\tau^{(\beta)}\in\Sigma_L$ of $\tau$, we have $\varphi|_{\tau^{(\alpha)}}\neq\varphi|_{\tau^{(\beta)}}$.
\end{definition}

Separability is preserved under localization.

\begin{proposition}\label{prop:strata_separable}
    If $\bb{L}$ is separable, the tropical Lagrangian multi-section $\bb{L}_{\tau'}$ is also separable.
\end{proposition}
\begin{proof}
Let $K_{\tau}(\sigma)\in\Sigma_{\tau}$. Suppose $K_{\tau'}(\sigma^{(\alpha)}),K_{\tau'}(\sigma^{(\beta)})$ are two distinct lift of $K_{\tau}(\sigma)$ and contain distinct lifts of $K_{\tau}(\sigma)$. In particular, $\sigma^{(\alpha)}\neq\sigma^{(\beta)}$. As we have seen, the difference $p_{\tau'}^*\varphi_{\tau'}-\varphi$ is an affine function on $\text{Star}^{\text{o}}(\tau')$. Thus
$$m(\sigma^{(\alpha)})-m(\sigma^{(\beta)})=p_{\tau}^*(m_{\tau'}(\sigma^{(\alpha)})-m_{\tau'}(\sigma^{(\alpha)})).$$
By separability, $m(\sigma^{(\alpha)})\neq m(\sigma^{(\beta)})$. Hence for any $v\in\tau$,
$$\left(m_{\tau'}(\sigma^{(\alpha)})-m_{\tau'}(\sigma^{(\beta)})\right)(p_{\tau}(v))=\left(m(\sigma^{(\alpha)})-m(\sigma^{(\beta)})\right)(v)\neq 0.$$
Hence $\bb{L}_{\tau'}$ is also separable.
\end{proof}

The relation between restriction and localization is given by the following

\begin{theorem}\label{thm:bundle_localization}
    Let $\cu{E}$ be a toric vector bundle on $X_{\Sigma}$ and $X_{\tau}\subset X_{\Sigma}$ be a toric strata. Let $\cu{E}_{\tau}:=\cu{E}|_{X_{\tau}}$. By choosing a lift of $p_{\tau}:N\to N/(\bb{R}\tau\cap N)$, $\cu{E}_{\tau}^{(\alpha)}$ admits a structure of toric vector bundle over $X_{\tau}$. Moreover, if $\cu{E}_{\tau}^{(\alpha)}$ is an indecomposable summand of $\cu{E}_{\tau}$, the associated tropical Lagrangian multi-section $\bb{L}_{\tau}^{(\alpha)}$ of $\cu{E}_{\tau}^{(\alpha)}$ is a localization of $\bb{L}_{\cu{E}}$ along a unique cone.
\end{theorem}
\begin{proof}
    Let $\iota_{\tau}:N/(\bb{R}\tau\cap N)\to N$ be a lift of the projection $p_{\tau}:N\to N/(\bb{R}\tau\cap N)$. Then it is easy to check that
    $$\ol{\lambda}\cdot v:=\iota_{\tau}(\ol{\lambda})\cdot v$$
    defines a toric vector bundle structure on $\cu{E}_{\tau}$ over $X_{\tau}$. It is by construction that the tropical Lagrangian multi-section $\bb{L}_{\tau}$ associated to $\cu{E}_{\tau}$ is a localization of $\bb{L}_{\cu{E}}$. Since $\cu{E}_{\tau}^{(\alpha)}$ is a summand, it must be toric by Corollary \ref{cor:shift}. There is an inclusion $\bb{L}_{\tau}^{(\alpha)}\subset\bb{L}_{\tau}$ which covering $|\Sigma|$. Hence $\bb{L}_{\cu{E}_{\tau}^{(\alpha)}}$ is also a localization of $\bb{L}_{\cu{E}}$. Separability of $\bb{L}_{\cu{E}}$ implies the cone that we localize is unique.
\end{proof}

\bibliographystyle{amsplain}
\bibliography{geometry}
	
\end{document}